\documentclass{article}

\usepackage[utf8]{inputenc}
\usepackage[english]{babel}

  \usepackage[T1]{fontenc}
  \usepackage[mathcal]{eucal}
  \usepackage{amsfonts}
 \usepackage[colorlinks=false]{hyperref}
  
  \usepackage{amsmath}
  \usepackage{amsthm}
  \usepackage{amssymb}
  \usepackage{mathtools}
  \usepackage{mathrsfs}
  \usepackage{stmaryrd}
  \usepackage{authblk}
  \usepackage{graphicx}
  \usepackage[all]{xy}
  \usepackage{relsize, exscale} 
  \usepackage{tikz-cd}
  \usepackage{wasysym}
  \usepackage{enumitem}
  \usepackage{float}
\usepackage{xcolor}

\makeatletter
\newcommand{\subjclass}[2][1991]{%
  \let\@oldtitle\@title%
  \gdef\@title{\@oldtitle\footnotetext{#1 \emph{Mathematics subject classification.} #2}}%
}
\newcommand{\keywords}[1]{%
  \let\@@oldtitle\@title%
  \gdef\@title{\@@oldtitle\footnotetext{\emph{Key words and phrases.} #1.}}%
}
  
\newtheorem{thm}{Theorem}[section]
\newtheorem{df}[thm]{Definition}

\newtheorem{lem}[thm]{Lemma}
\newtheorem{prop}[thm]{Proposition}
\newtheorem{cor}[thm]{Corollary}
\newtheorem{nota}[thm]{Notation}
\newtheorem{rem}[thm]{Remark}

\newtheorem{set}[thm]{Setup}

\newcommand{\bea}{\begin{eqnarray}}
\newcommand{\eea}{\end{eqnarray}}
\newcommand{\bna}{\begin{eqnarray*}}
\newcommand{\ena}{\end{eqnarray*}}

\newcommand{\R}{\mathbb{R}}
\newcommand{\Bl}{\operatorname{Bl}}
\newcommand{\Op}{\mathcal{O}_{\mathbb{P}^1}}

\newcommand{\ol}{\mathcal{O}}
\newcommand{\N}{\mathbb{N}}

\newcommand{\Ext}{\operatorname{Ext}}

\newcommand{\pr}{\mathbb{P}}
\DeclareMathOperator{\Supp}{Supp}

\DeclareMathOperator{\NE}{NE}
\newcommand{\Pic}{\operatorname{Pic}}
\newcommand{\NS}{\operatorname{NS}}
\DeclareMathOperator{\Fix}{Fix}
\DeclareMathOperator{\Mob}{Mob}

\title{\textbf{Rationally connected threefolds with nef and bad anticanonical divisor, II}}
\author{Zhixin Xie}
\date{}
\subjclass[2010]{14E30, 14M22.}

\begin{document}
\maketitle
\begin{abstract}
Let $X$ be a smooth complex projective rationally connected threefold with $-K_X$ nef and not semi-ample. 
   In our previous work, we classified all such threefolds when $|{-}K_X|$ has no fixed divisor. In this paper, we continue our classification when $|{-}K_X|$ has a non-zero fixed divisor.
\end{abstract}

\section{Introduction}
Complex projective Fano manifolds play an important role in the framework of the Minimal Model Program (MMP) and appear as one of the building blocks in the birational classification of varieties. Fano manifolds are classified up to dimension three – the classification of Fano threefolds was achieved by Mori-Mukai and Iskovskikh, see \cite{MR641971} and \cite{MR463151, MR503430}. 

Complex projective manifolds with nef anticanonical divisor are a natural generalisation of Fano manifolds and one hopes to similarly fulfil a complete classification for this class of manifolds. One of the methods to study such a manifold is the decomposition theorem for projective manifolds with nef anticanonical bundle by Cao and H\"{o}ring \cite{MR3959071}: for such a manifold $X$, its universal cover $\widetilde{X}$ decomposes as a product
\[
\widetilde{X}\simeq \mathbb{C}^q\times\prod Y_j\times\prod S_k\times Z,
\]
where $Y_j$ are irreducible projective Calabi-Yau manifolds, $S_k$ are irreducible projective hyperk\"{a}hler manifolds (so that $Y_j$ and $S_k$ have trivial canonical bundle), and $Z$ is a projective rationally connected manifold with $-K_Z$ nef (and non trivial as $Z$ is rationally connected). In view of this result, it is important to study the case when X is rationally connected, and it is also the most difficult one.

Recent results by Birkar, Di Cerbo and Svaldi \cite[Theorem~1.6]{Birkar2020BoundednessOE} showed that birationally, there are only finitely many deformation families of projective rationally connected threefolds with nef anticanonical divisor. Thus it is in principle possible to classify these varieties.

This paper addresses the classification problem of smooth projective rationally connected threefolds $X$ with $-K_X$ nef and not semi-ample. In our earlier work \cite{Xie20}, the case when $|{-}K_X|$ has no fixed divisor is completely classified. In this paper, we consider the remaining case when $|{-}K_X|$ has a non-zero fixed divisor. It is shown in op.\ cit.\ that a general member of the mobile part of $|{-}K_X|$ has at lease two irreducible components. Our main result
gives a complete classification when an irreducible component is a non-rational surface and shows that in this case, a general member of the mobile part of $|{-}K_X|$ has exactly two irreducible components.

\subsection{Previous work}
Let $X$ be a smooth projective rationally connected threefold with $-K_X$ nef. If $-K_X$ is semi-ample, we refer to \cite[Sections 5, 6]{MR2129540} for a partial classification.
Another approach to obtain a classification in this case stems from its similarity with the case of weak Fano threefolds, i.e.\ threefolds with nef and big anticanonical bundle. One may analyse the plurianticanonical morphism
\[
\phi_{|{-}mK_X|}\colon X\to Y
\]
for $m$ sufficiently large, as done in the weak Fano case, which led to boundedness of weak Fano threefolds, see \cite{MR1837167,MR1771144,Mckernan2002BoundednessOL}. Together with a discussion of the Mori contractions, one may obtain a classification by following the strategy in \cite{MR2198800,MR2784025}, where the authors gave a classification of weak Fano threefolds with Picard number two.
 
We thus focus on the case where $-K_X$ is not semi-ample.
In \cite{MR2129540}, Bauer and Peternell gave the following criterion for verifying the non semi-ampleness.
\begin{thm}{\rm (\cite[Theorem 2.1]{MR2129540})}
Let $X$ be a smooth projective rationally connected threefold with $-K_X$ nef. Then the Iitaka dimension $\kappa(X, -K_X)$ is at least $1$. 

If the nef dimension \footnote{See \cite[Theorem 2.1]{MR1922095} for the definition of nef dimension and the construction of the nef reduction map associated to a nef divisor.} $n(X, -K_X)$ is $1$ or $2$, then $-K_X$ is semi-ample and the nef reduction map associated to $-K_X$ can be taken as the Stein factorisation of the map defined by some positive multiple of $-K_X$ which is globally generated.
\end{thm}
\noindent
By a result of Kawamata \cite[Theorem~6.1]{MR782236}, if $\kappa(X,-K_X)=\nu(X,-K_X)$, where $\nu$ denotes the numerical dimension, then $-K_X$ is semi-ample. Thus, in practice, the above theorem, together with \cite[Theorem 6.1]{MR3145741}, implies that the non semi-ampleness of $-K_X$ is equivalent to $n(X,-K_X)=3$ and $\nu(X,-K_X)=2$, which is also equivalent to $\nu(X,-K_X)=2$ and $\kappa(X,-K_X)=1$.

We start the investigation with the base locus of the anticanonical system $|{-}K_X|$ as the latter one is non-empty and not semi-ample. When $|{-}K_X|$ has no fixed divisor, a complete classification is obtained in \cite[Theorem 1.1]{Xie20}. Now if $|{-}K_X|$ has a non-zero fixed divisor, it turns out that, after a finite
sequence of flops, one can assume that the mobile part of $|{-}K_X|$ is base-point-free.
\begin{thm}{\rm(\cite[Theorem 1.2, Corollary 2.8, Lemma 4.2]{Xie20})}\label{introfixed}
    Let $X$ be a smooth projective rationally connected threefold with $-K_X$ nef and not semi-ample. Assume that $\Fix|{-}K_X|\neq 0$. Then there exists a finite sequence of flops $\psi\colon X\dashrightarrow X'$ such that the following holds: 
\begin{itemize}
\item $X'$ is smooth,
    \item $-K_{X'}$ is nef,
    \item $\Mob|{-}K_{X'}|$ is base-point-free and induces a fibration $f\colon X'\to\mathbb{P}^1$.
\end{itemize}
Moreover, we have
\[
|{-}K_{X'}|=A+|kF| \text{ with } k\geq 2,
\]
where $F$ is a general fibre of $f$. Furthermore, we have
\[A^3=A^2\cdot F=0,\] 
and $F$ is a smooth surface with $-K_F$ effective, nef and not semi-ample.
\end{thm}

Now, in order to study the geometry of the fibration $X'\to \pr^1$ in Theorem \ref{introfixed}, we consider the following setup where we denote $X'$ by $X$ for simplicity of notation in the rest of our discussion.
\begin{set}\label{generalsetup}
Let $X$ be a smooth projective rationally connected threefold with anticanonical bundle $-K_X$ nef and not semi-ample. Assume that $A\coloneqq \Fix|{-}K_X|\neq 0$ and $|B|\coloneqq\Mob|{-}K_X|$ is base-point-free, inducing a fibration $f\colon X\to\mathbb{P}^1$.

If $F$ is a fibre of $f$, then $F$ is a smooth surface with $-K_F$ effective, nef and not semi-ample. We have
\[|{-}K_X|=A+|kF| \text{ with } k\geq 2\] 
and \[A^3=A^2\cdot F=0.\] 

Now we write $A=A_h+A_v$, where $A_h$ and $A_v$ are effective divisors such that $A_h|_F=-K_F$ and $A_v|_F=0$ for a general fibre $F$.
\end{set}

\subsection{Main results and organisation of the paper}
In this paper, we give a complete classification when the general fibre $F$ in Setup \ref{generalsetup} is a non-rational surface.
\begin{thm}\label{mainnonrat}
{\rm(A)} In Setup \ref{generalsetup}, assume that the surface $F$ is non-rational.
Then $X=\pr(\mathcal{V})$ is a $\pr^1$-bundle over $Y$, where $Y$ is isomorphic to $\pr^2$ blown up in $9$ points such that $-K_Y$ is nef and base-point-free (thus induces an elliptic fibration $\pi\colon Y\to\pr^1$ with general fiber denoted by $R$) and $\mathcal{V}$ is a rank-$2$ vector bundle defined by a non-split extension
\[
0\to\mathcal{O}_Y\to\mathcal{V}\to\mathcal{O}_Y(K_Y)\to 0,
\]
and the fibration $f$ factors as $X\rightarrow Y\overset{\pi}\rightarrow\pr^1$.

Furthermore, $|{-}K_X|=2D+|2F|$ with $D=\pr\big ( \mathcal{O}_Y(K_Y)\big ) \simeq Y$, and $F=\mathbb{P}\big( \mathcal{E}\big )$ is a $\mathbb{P}^1$-bundle over the smooth elliptic curve $R$, where $\mathcal{E}$ is a rank-$2$ vector bundle over $R$ defined by a non-split extension
\[
0\to\mathcal{O}_R\to\mathcal{E}\to\mathcal{O}_R\to 0.
\]

\smallskip
{\rm(B)} Conversely, let $Y$ be $\pr^2$ blown up at $9$ points such that $-K_Y$ is nef, base-point-free and thus induces an elliptic fibration $\pi\colon Y\to\pr^1$ with general fibre denoted by $R$. Let $\mathcal{V}$ be a rank-$2$ vector bundle over $Y$ defined by a non-split extension
\begin{equation}\label{extennonrat}
0\to\mathcal{O}_Y\to\mathcal{V}\to\mathcal{O}_Y(K_Y)\to 0
\end{equation}
and let $\varphi\colon X\coloneqq\pr(\mathcal{V})\to Y$. Then $-K_X$ is nef and not semi-ample and $|{-}K_X|=2D+|2F|$, where $D\coloneqq\pr(\mathcal{O}_Y(K_Y))$ and $F$ is a general fibre of the fibration $f\coloneqq \pi\circ\varphi\colon X\to\pr^1$. Moreover, $F=\pr(\mathcal{E})$ is a $\pr^1$-bundle over the smooth elliptic curve $R$, where $\mathcal{E}$ is a rank-$2$ vector bundle over $R$ and defined by a non-split extension
\[
0\to\mathcal{O}_R\to\mathcal{E}\to\mathcal{O}_R\to 0.
\]
\end{thm}

Since we applied Theorem \ref{introfixed} in order to reduce to Setup \ref{generalsetup}, the classification obtained in Theorem \ref{mainnonrat} is up to flops. Thus we also want to track, a posteriori, the sequence of flops. To this end, we describe all extremal rays of the cone $\overline{\NE}(X)\cap K_X^{\perp}$ for the threefolds $X$ in Theorem \ref{mainnonrat}. We obtain more precisely the following result.
\begin{prop}\label{maincone}
   In Theorem \ref{mainnonrat}, we have a morphism $\NE(D)\to\NE(X)$ induced by the inclusion $D\hookrightarrow X$. The cone $\NE(X)$ is closed and the extremal rays of the subcone $\NE(X)\cap K_X^{\perp}$ are spanned by the classes of $(-1)$-curves on $D$ and the classes of $(-2)$-curves on $D$ (or the class $-K_D$ if there is no $(-2)$-curve on $D$, i.e.\ if the elliptic fibration $f|_D\colon D\to\pr^1$ has no reducible fibre). 
   
   Moreover, there are infinitely many flopping contractions on $X$ and each flopping contraction contracts an extremal ray spanned by a class of a $(-1)$-curve on $D$.
\end{prop}

\medskip

\noindent
{\bf Plan.} We briefly explain the organisation of the paper. In Section \ref{sec:prelim}, we discuss some general results about the geometry of threefolds $X$ as in Setup \ref{generalsetup}.

Section \ref{sec:classification} is devoted to the proof of our main Theorem \ref{introfixed}. Under Setup \ref{generalsetup} and assuming that a general fibre $F$ of $f\colon X\to\pr^1$ is a non-rational surface, we start by investigating the structure of the fixed divisor $A$ of the anticanonical system $|{-}K_X|$ in Subsection \ref{subsec:fixedpart}. Arguing by contradiction, we will show that the fixed divisor $A$ has no $f$-vertical part, and we describe more precisely the $f$-horizontal part of $A$ in Proposition \ref{nonratfixedpart}. This leads to further restrictions on the geometry of $X$ when we run the MMP in Subsection \ref{subsec:mmp} to obtain the classification in Theorem \ref{introfixed}.

In Section \ref{sec:cone}, we study the cone of effective curves of $X$. To this end, we describe all the $K_X$-trivial curves by first showing that every $K_X$-trivial curve class is proportional to a curve (class) contained in the fixed divisor of $|{-}K_X|$. Since the fixed divisor (with reduced structure) is isomorphic to the blow-up of $\pr^2$ at the $9$ base points of a cubic pencil, its cone of effective curves is classically known. In this way, we describe all extremal rays of the cone $\overline{\NE}(X)\cap K_X^{\perp}$ and the corresponding extremal contractions in Lemma \ref{lem:Ktrivcurve}, which implies directly Proposition \ref{maincone}.

\medskip
{\bf Acknowledgements.} Some results in this paper are based on my PhD thesis. I would like to express my sincere gratitude to my supervisor, Andreas H\"{o}ring, for his patient guidance, his constant encouragements and his careful proof-reading.
I heartily thank Cinzia Casagrande for interesting discussions and inspiring suggestions. I thank Vladimir Lazi\'{c} and Nikolaos Tsakanikas for useful comments on the paper.

\section{Preliminary}\label{sec:prelim}

In this paper we work over the field $\mathbb{C}$.

\subsection{Notation and terminology}
We use the following notation throughout the paper, see \cite[Definition 2.1.3, Remark 2.3.17]{MR2095471} for definitions.
\begin{nota}\label{numericaldim}
Let $X$ be a normal projective variety and $D$ a Cartier divisor on $X$. We denote by
\begin{itemize}
    \item $\kappa(D,X)$ the Iitaka dimension of $D$.
    \item $\nu(D,X)\coloneqq \max \{n\mid D^n\not\equiv 0\}$ the numerical dimension of $D$ when $D$ is nef.
\end{itemize}
Consider the complete linear system $|D|$. We denote by
\begin{itemize}
    \item $\Fix|D|$ the fixed divisor of $|D|$.
    \item $\Mob|D|=|D|-\Fix|D|$ the mobile part of $|D|$.
\end{itemize}
\end{nota}

Note that the numerical dimension can also be defined for a pseudo-effective divisor, see for example \cite[Chapter V, Definition 2.5]{MR2104208}.

\begin{df}
    Let $X$ be a normal projective variety. A flopping contraction is an extremal birational contraction $f\colon X\to Y$ to a normal variety $Y$ such that the exceptional locus of $f$ has codimension at least two in $X$ and $K_X$ is numerically $f$-trivial. 

    If additionally $D$ is a $\mathbb{Q}$-Cartier divisor on $X$ such that $-(K_X+D)$ is $f$-ample, then the $(K_X+D)$-flip of $f$ is called the $D$-flop.
\end{df}

\subsection{Results on surfaces}
We will need the following results on surfaces.
\begin{lem}{\rm(\cite[Lemma 4.4, Corollary 4.6]{Xie20})}\label{surfacetwocpnts}
Let $S$ be a projective Gorenstein surface such that the anticanonical divisor $-K_S$ is of the form:
\[
-K_S\sim D_1+D_2,
\]
where $D_1$ is effective and $D_2$ is a non-zero effective Cartier divisor which is nef and divisible by $r\geq 2$ in $\NS(S)$.

Suppose that $D_2^2=0$, and that one of the following assertions holds:
\begin{enumerate}[label=\normalfont(\roman*)]
    \item $S$ is not covered by $D_2$-trivial curves;
    \item $D_2$ contains a smooth curve of positive genus.
\end{enumerate}
Then $D_1=0$ and the surface $S$ is a $\mathbb{P}^1$-bundle over a smooth elliptic curve.
\end{lem}

\begin{lem}\label{surfaceproduct}
Let $S$ be a ruled surface over a smooth elliptic curve $B$. Suppose that $S$ admits an elliptic fibration $\tau\colon S\to\pr^1$. If $h^0\big (S,\mathcal{O}_S(-K_S)\big )\geq 3$, then $S\simeq B\times\pr^1$.
\end{lem}
\begin{proof}
Since $S$ is a $\mathbb{P}^1$-bundle over a smooth elliptic curve $B$ and $S$ admits an elliptic fibration, we have $S=\mathbb{P}(\mathcal{V})$ by \cite[Theorem 5]{MR242198}, where $\mathcal{V}$ is one of the following:
\begin{enumerate}[label=\normalfont(\alph*)]
    \item $\mathcal{V}$ is the unique indecomposable rank-$2$ vector bundle of degree $1$ on $B$;
    \item $\mathcal{V}=\mathcal{O}_{B}\oplus \mathcal{L}$, where $\mathcal{L}$ is a (possibly trivial) torsion line bundle.
\end{enumerate}
In the first case, let $\ell$ be a fibre of the ruling and let $\theta_i$ be a section with minimal self-intersection, i.e.\ $\theta_i^2=1$. Then $-K_{S}\sim 2\theta_i-\ell$. By \cite[Theorem 5(iii)]{MR242198}, the elliptic fibration is given by the linear system $|4\theta_i-2\ell|$. Hence,
\[
h^0\big ( S,\mathcal{O}_{S}(-K_{S})\big ) =1.
\]
In the second case, we have
\[
h^0\big ( S,\mathcal{O}_{S}(-K_{S})\big ) =h^0(B, S^2\mathcal{V}\otimes\mathcal{L}^*)=h^0(B,\mathcal{L}\oplus\mathcal{L}^*\oplus\mathcal{O}_{B})\leq 3,
\]
with equality if and only if $\mathcal{L}=\mathcal{O}_{B}$.

Since $h^0\big ( S,\mathcal{O}_S(-K_S)\big ) \geq 3$ by assumption, we obtain $\mathcal{V}=\mathcal{O}_B\oplus\mathcal{O}_B$. Thus $S\simeq B\times\pr^1$.
\end{proof}

\subsection{Results on threefolds}
We will first prove the following two lemmas which give geometric restriction on the threefold with nef anticanonical divisor, when we consider certain types of extremal contractions on the threefolds.
\begin{lem}\label{contractfixedcpnt}
Let $X$ be a smooth projective threefold with $-K_X$ nef. Let $A\coloneqq \Fix|{-}K_X|$. Consider a $K_X$-negative extremal contraction $\varphi\colon X\to Y$. Assume that $\varphi$ is a divisorial contraction which contracts an irreducible component of $A$ to a smooth curve or a smooth point. Then $\kappa(Y,-K_Y)=\kappa(X,-K_X)$ and $\nu(Y,-K_Y)=\nu(X,-K_X)$.
\end{lem}
\begin{proof}
Denote by $E$ the exceptional divisor of $\varphi$. Then
\[
\varphi^*(-K_Y)=-K_X+mE
\]
with $m=1$ or $2$.
Since $\ol_X(E)\xhookrightarrow{} \ol_X(A) \xhookrightarrow{} \ol_X(-K_X)$, one has
\[
\kappa(X,-K_X)\leq \kappa(X,-K_X+mE)\leq \kappa(X,-(m+1)K_X)=\kappa(X,-K_X).
\]
By \cite[Proposition V.2.7(1)]{MR2104208}, one has
\[
\nu(X,-K_X)\leq \nu(X,-K_X+mE)\leq \nu(X,-(m+1)K_X)=\nu(X,-K_X).
\]
Hence, $\kappa(Y,-K_Y)=\kappa(X,-K_X)$ and $\nu(Y,-K_Y)=\nu(X,-K_X)$.
\end{proof}

\begin{lem}\label{flopk}
Let $X$ be a smooth projective threefold with $|{-}K_X|=A+|kF|$, where $A\coloneqq \Fix|{-}K_X|\neq 0$, $F$ is a prime divisor, and $k\geq 2$ is an integer. Suppose that there exists an $\epsilon A$-flop, where $\epsilon>0$ satisfies that the pair $(X,\epsilon A)$ is log-canonical. Then $A$ has multiplicity at least $k$ along the flopping curve.
\end{lem}
\begin{proof}
By assumption, there exists an $\epsilon A$-flop:
\[
\psi\colon X \dasharrow X^+,
\]
where $X^+$ is again smooth by \cite[Theorem 2.4]{MR986434}. Since $\psi$ induces an isomorphism in codimension one, the anticanonical system $|{-}K_{X^+}|$ has a non-empty fixed divisor $A^+\coloneqq\psi_*(A)$ and we can write
\[
|{-}K_{X^+}|=A^+ + |kF^+|,
\]
where $F^+\coloneqq\psi_*(F)$ and $|kF^+|$ is the mobile part of the anticanonical system.

Since $\psi$ is a flop, there exists a common resolution:
\begin{center}
\begin{tikzcd}
                             & \widetilde{X} \arrow[ld, "g"'] \arrow[rd, "h"] &     \\
X \arrow[rr, "\psi", dashed] &                                            & X^+
\end{tikzcd} 
\end{center}
such that $g^*(K_X)=h^*(K_{X^+})$. Moreover, by \cite[Proposition 5-1-11]{MR946243}, one has
\begin{equation}\label{eq:flop}
K_{\tilde{X}}=g^*(K_X+\epsilon A) + \sum_i a_i E_i = h^*(K_{X^+}+\epsilon A^+) + \sum_i a_i^+ E_i,
\end{equation}
where $a_i^+\geq a_i$, and $a_i^+>a_i$ if and only if $g(E_i)$ is contained in the flopping locus.

Since $|{-}K_X|= A+|kF|$ and $|{-}K_{X^+}|=A^+ + |kF^+|$, the equality (\ref{eq:flop}) gives
\[
h^*(kF^+)-g^*(kF) = (1-\epsilon)\big (g^*(A)-h^*(A^+)\big )+\sum_i(a_i^+-a_i)E_i = \frac{1}{\epsilon}\sum_i(a_i^+-a_i)E_i
\]
is effective.
Since $F$ and $F^+$ are Cartier divisors, we can write \[h^*(kF^+)-g^*(kF) =\sum_i kn_i E_i\] with $n_i\in\N$.

Since $g^*(A)=h^*(A^+)+h^*(kF^+)-g^*(kF)$, we obtain that \[g^*(A)-\tilde{A}-\sum_i kn_i E_i\] is effective, where $\tilde{A}\coloneqq g_*^{-1}(A)$ is the strict transform of $A$ in $\tilde{X}$. Therefore, $A$ has multiplicity at least $k$ along the flopping curve.
\end{proof}

More generally, we will need the following result by Wilson on the classification of crepant contractions of an extremal ray on a threefold.
\begin{prop}{\rm(\cite[Theorem 2.2]{MR1150602},\cite{MR1235026}; \cite[Proposition 3.1]{MR1472891})}\label{wilson}
Let $X$ be a smooth projective threefold and let $\phi\colon X\to Y$ be a crepant contraction of an extremal ray, contracting some irreducible surface $E\subset X$ to a curve $C\subset Y$. Then $C$ is a smooth curve and $\phi\colon E\to C$ is a conic bundle over $C$ such that one of the following holds:
\begin{enumerate}[label=\normalfont(\roman*)]
    \item $E$ is normal and a general fibre of $\phi\colon E\to C$ is a smooth conic;
    \item $E$ is non-normal and a general fibre of $\phi\colon E\to C$ is two lines meeting at one point.
\end{enumerate}
For a general fibre $l$ of $\phi\colon E\to C$, one has $E\cdot l=-2$. A singular fibre of $\phi\colon E\to C$ is either two $\pr^1$'s intersecting at one point, or a double line.

Furthermore, if $E$ is normal, then the possible singularities of $E$ are $A_n$ singularities at the point where distinct components of a singular fibre meet, or $A_1$ singularities appearing as a pair on some double fibre. 
\end{prop}

\subsection{General setup and basic results}
Let us first point out an important special case under our Setup \ref{generalsetup}. 
\begin{lem}\label{lem:loctriv}
    In Setup \ref{generalsetup}, suppose that the relative anticanonical divisor $-K_{X/\pr^1}$ is nef. Then \[X\simeq F\times \pr^1.\]
\end{lem}
\begin{proof}
Since $-K_{X/\pr^1}$ is nef, the fibration $f$ is locally trivial in the Euclidean topology by \cite[Theorem A.12]{Patakfalvi2019OnTB} and \cite[Proposition 2.8]{MR4057779}. Applying the latter proposition, we further obtain $X\simeq F\times \pr^1$; we explain how to apply the proposition in our case, as follows.

Note that in \cite[Proposition 2.8]{MR4057779} the assumption is different. Instead of assuming

(H1): \textit{the relative anticanonical divisor is nef}, 

\noindent
they assume

(H2): \textit{there exists an $f$-very ample line bundle $L$ such that $f_*(mL)$ is a numerically flat vector bundle for every integer $m\leq 1$.} 

\noindent
However, when the fibration is over a smooth curve, \cite[Proposition A.11]{Patakfalvi2019OnTB} shows that (H1) implies (H2). Since in our case, the fibration is over $\pr^1$ which is simply connected, \cite[Proposition 2.8]{MR4057779} gives $X\simeq F\times \pr^1$.
\end{proof}

\begin{lem}\label{verticalcurveinAh}
In Setup \ref{generalsetup}, assume that $A_h$ is $f$-relatively nef. Then for any $f$-vertical curve $\ell$ contained in $A_h$, one has $K_X\cdot \ell= A_h\cdot \ell= 0$.
\end{lem}
\begin{proof}
Suppose by contradiction that there exists a $f$-vertical curve $\ell\subset A_h$ such that $A_h\cdot \ell>0$. 

   Let $F_0$ be the special fibre of $f$ which contains $\ell$. Since
   for a general fibre $F$ of $f$,
   \[A_h\cdot (F|_{A_h})=(A_h|_F)^2=0,\] we have that $A_h\cdot(F_0|_{A_h})=0$. Hence we can write
   \[F_0 \cap A_h= m\ell + \ell' \]
   with $m>0$ and $A_h\cdot \ell'<0$. This contradicts the fact that $A_h$ is $f$-relatively nef.
\end{proof}

Now in Setup \ref{generalsetup}, consider a $K_X$-negative extremal contraction $\varphi\colon X\to Y$.
Let $\Gamma$ be an extremal ray contracted by $\varphi$. Recall that the length of an extremal ray $\Gamma$ is defined by
\[
l(\Gamma)=\min\{-K_X\cdot Z\mid [Z]\in \Gamma\}.
\]
Let $\ell$ be a rational curve such that $[\ell]\in \Gamma$ and $-K_X\cdot \ell =l(\Gamma).$ If $\varphi$ is birational, we denote the exceptional divisor of $\varphi$ by $E$. In the remainder of this section, we will describe all the possible contractions $\varphi$.

\subsubsection{Birational extremal contractions}
We first describe all the divisorial  contractions of a $K_X$-negative extremal ray on the threefold $X$ in Setup \ref{generalsetup}.
\begin{prop}\label{generalbiratcontr}
In Setup \ref{generalsetup}, assume that $A_v=0$. Let $\varphi\colon X\to Y$ be a divisorial $K_X$-negative extremal contraction and let $E$ be the exceptional divisor. Then $\varphi$ is the blow-up of a smooth curve $C$ in $Y$ and $E$ is a ruled surface. Denote by $C_0$ a canonical section of the tautological line bundle $\ol_{\pr(V)}(1)$, where $V$ is the normalisation of $N^*_{C/Y}$. Then one of the following cases occurs.
\begin{enumerate}[label=\normalfont(\roman*)]
    \item The divisor $E$ is an irreducible component of $A$ and $\varphi$ contracts $E$ horizontally (i.e.\ every fibre of $\varphi$ is an $f$-horizontal curve) to a smooth curve.
    
    \item The fibration $f$ factors as $f=f'\circ\varphi$, which gives a fibration $f'\colon Y\to \mathbb{P}^1$. Let $A_Y\coloneqq \varphi(A)$ and $F_Y\coloneqq \varphi(F)$. We are in one of the following cases:
    \begin{enumerate}[label=\normalfont(\alph*)]
        \item $f'$ maps the blow-up curve $C$ onto $\pr^1$, and $\varphi|_F\colon F\to F_Y$ is the blow-down of some $(-1)$-curves in $F$. In this case,  $A_Y\simeq A$, $-K_{F_Y}$ is nef and big, and $-K_Y$ is nef and big.
        \item $f'$ maps the blow-up curve $C$ to a point. In this case, $A|_E\equiv C_0$, $A_Y\simeq A$ and $F_Y\simeq F$.
        
    \end{enumerate}
\end{enumerate}
\end{prop}
\begin{proof}
{\bf Case $A\cdot\ell =0$.} 
In this case, we have $F\cdot \ell=1$ and $-K_X\cdot \ell=A\cdot \ell+kF\cdot \ell=k=2$. Hence $\varphi$ is the blow-up of a smooth point on $Y$, with exceptional divisor $E\simeq\mathbb{P}^2$. As $E$ is not fibered, it is contained in a fiber of $f$ and thus $F\cdot E=0$. This contradicts the fact that $F\cdot \ell=1$.

\medskip

{\bf  Case $A\cdot \ell<0$.} In this case, we have $F\cdot \ell>0$ and $E$ is an irreducible component of $A$ as the contraction is divisorial. Moreover, as a curve in $A\cdot F$ is $K_X$-trivial, we obtain that $\varphi$ contracts $E$ horizontally to a curve.

\medskip

{\bf Case $A\cdot \ell >0$.} 
In this case, $F\cdot \ell=0$ since otherwise $-K_X\cdot \ell>2$, which contradicts the classification of Mori, see \cite[Section 3]{MR715648}.

\medskip
{\em Step 1.} We show in this step that $E$ is not contracted to a point.

Suppose by contradiction that $E$ is contracted to a point. Since $F\cdot \ell=0$, we have $F\cdot E=0$. Thus $E$ is contained in some fibre of $f$.

Note that $E$ is not an irreducible component of $A$, otherwise $E$ is contracted to a curve. Hence $A\cdot E$ is a non-zero effective $1$-cycle contained in some fibre of $f$.  
By Lemma \ref{verticalcurveinAh}, $A\cdot(A\cdot E)=0$ and thus $A\cdot\ell=0$. This contradicts the fact that $A\cdot\ell>0$.
\medskip

{\em Step 2.} By {\em Step 1}, we deduce that $\varphi$ contracts $E$ to a curve $C$. Then $F\cdot\ell=0$, $A\cdot\ell=1$ and $E=\mathbb{P}(N_{C/Y}^*)$ is a ruled surface. 
Let $V=N_{C/Y}^*\otimes\mathcal{L}$ with $\mathcal{L}\in \Pic(C)$ be a normalisation of $N_{C/Y}^*$, see \cite[Chapter V, Proposition 2.8]{MR0463157}. Let $\mu\coloneqq \deg\mathcal{L}$ and let $C_0$ be a canonical section of the tautological line bundle $\mathcal{O}_{\mathbb{P}(V)}(1)$ such that $C_0^2=-e = c_1(V)$.
Then
\[
-K_X|_E \equiv C_0 + m\ell
\]
with $m\in\N$,
\[
N_{E/X}\equiv -C_0 + \mu \ell,
\]
\[
K_Y\cdot C = -m - \mu,
\]
and $-K_Y$ is nef if and only if $m+\mu \geq 0$.

Let $g$ be the genus of the curve $C$. Then by the adjunction formula,
\[
K_E^2 = \big( (K_X + E)|_E \big) ^2,
\]
and thus $8(1-g)= 4m - 4e - 4\mu$, i.e.\ $\mu = m - e + 2g -2. $

Since $(-K_X)^3=0$, we have
\[
(-K_Y)^3 = 3(-K_X|_E)^2 + 3K_Y\cdot C - 2E^3 = 2g -2 + 2(2m-e).
\]
As $-K_X|_E$ is nef, we have $(-K_X|_E)\cdot C_0\geq 0$, i.e.\ $m\geq e$.

We have two different subcases: either $f|_E\colon E\to\pr^1$ is surjective, or $f(E)$ is a point on $\pr^1$.

\medskip
{\em Step 3.}
 In this step, we consider the case when $f|_E\colon E\to\pr^1$ is surjective. Then $E\cdot F$ is a non-zero effective $1$-cycle. 
Note that $E$ is not an irreducible component of $A$, as $\ell$ is $f$-vertical and $A\cdot \ell >0$ by assumption.

Since $F\cdot\ell=0$, one has $E\cdot F=b\ell$ with $b\in\N^*$, and a fibre of $\varphi$ must be contained in some fibre of $f$. In particular, by the rigidity lemma \cite[Lemma 1.15]{debarre2013higher}, the fibration $f$ factors as $f=f'\circ\varphi$ such that $f'\colon Y\to \mathbb{P}^1$, and the morphism $\varphi$ contracts $E$ to a smooth curve $C\subset Y$ which surjects to $\mathbb{P}^1$ by $f'$. Restricted to a general fibre $F$, we have that $\varphi|_F$ blows down some $(-1)$-curves in $F$ and each of them meets the unique member in $|{-}K_F|$ transversally at one point.

\medskip

{\em Claim.} $-K_Y$ is nef and big.

Since $-K_X|_E\sim A|_E + kF|_E \equiv A|_E + kb\ell$ and $-K_X|_E \equiv C_0 + m\ell$, we have
\[
A|_E \equiv C_0 + (m-kb)\ell.
\]

Since $A\cdot E$ is an effective $1$-cycle and $-K_X|_E$ is nef, we have
$(-K_X|_E)\cdot (A|_E) \geq 0$. Hence,
\[
(C_0 + m\ell)\cdot \big ( C_0 + (m-kb)\ell\big ) \geq 0,
\]
i.e.\ $2m -e \geq kb$. We obtain
\[
m + \mu = 2m - e + 2g - 2 \geq kb - 2 \geq 0,
\]
and thus $-K_Y$ is nef. Moreover,
\[
(-K_Y)^3 = 2g - 2 + 2(2m - e) \geq -2 + 2kb \geq 2.
\]
Therefore, $-K_Y$ is nef and big. This proves the claim.
\medskip

{\em Step 4.}
 In this step, we consider the case when $f(E)$ is a point. Then $E$ is contained in a special fiber $F_0$ of $f$ and $E\cdot F=0$. Hence $\varphi$ is an isomorphism outside $F_0$. Moreover, the fibration $f\colon X\to \mathbb{P}^1$ factors as $f=f'\circ \varphi$ such that $f'\colon Y\to\mathbb{P}^1$.

We have
\[
A|_E\sim -K_X|_E \equiv C_0+m\ell.
\]
Since $E$ is contained in some fiber of $f$, $A\cdot E$ is an effective $f$-vertical $1$-cycle. Hence $A\cdot(A\cdot E)=0$ by Lemma \ref{verticalcurveinAh}. It remains to prove $m=0$.

Suppose that $m>0$. Since
\[
0=A\cdot(C_0+m\ell)=A\cdot C_0 + m,
\]
one has $A\cdot C_0<0$. Hence the curve $C_0$ is contained in $A$. But $C_0$ is $f$-vertical, this contradicts Lemma \ref{verticalcurveinAh} and proves $m=0$.
\end{proof}

\subsubsection{Non-birational contractions}\label{nonbiratsubsect}
Finally, we describe all the contractions of fibre type of a $K_X$-negative extremal ray on the threefold $X$ in Setup\ref{generalsetup}.
\begin{prop}\label{generalnonbiratcontr}
In Setup \ref{generalsetup}, let $\varphi\colon X\to Y$ be a non-birational $K_X$-negative extremal contraction. Then $Y$ is a smooth rational surface, and one of the following cases occurs:
 \begin{enumerate}[label=\normalfont(\roman*)]
 \item $Y\simeq F$ and $X\simeq F\times\pr^1$;
 \item $\varphi\colon X\to Y$ is a conic bundle and $f$ factors as $f\colon X\xrightarrow{\varphi}Y\rightarrow\mathbb{P}^1$.
 \end{enumerate}
\end{prop}
\begin{proof}
{\bf Case $\dim Y=2$.} 
In this case, we have $-K_X\cdot \ell\in \{1,2\}$, the morphism $\varphi\colon X\to Y$ is a conic bundle, and $Y$ is a smooth rational surface. Since $[\ell]$ is a movable class, we have $A\cdot \ell\geq 0$.

 We first consider when $F\cdot \ell>0$. Then $F\cdot \ell=1$ and $-K_X\cdot \ell=A\cdot \ell+kF\cdot \ell\geq 2$. We deduce that $k=2$, the morphism $\varphi$ is a $\mathbb{P}^1$-bundle and $\varphi|_F$ is birational. Consider the product map $p\coloneqq f\times\varphi\colon X\to \mathbb{P}^1\times Y$ which is generically one to one.

        \medskip
        {\em Claim.} $p$ is an isomorphism.
        
        Suppose that there exists a curve $R\subset X$ contracted by $p$, then $R$ is also contracted by $\varphi$. Hence $F\cdot R>0$ and $R$ is a fibre $\varphi^{-1}(y_0)$ of $\varphi$, where $y_0\in Y$. By the rigidity lemma \cite[Lemma 1.15]{debarre2013higher}, there exists a neighbourhood $Y_0\subset Y$ of $y_0$ and a factorisation $f|_{\varphi^{-1}(Y_0)}\colon\varphi^{-1}(Y_0)\xrightarrow{\varphi} Y_0\to\mathbb{P}^1$, which implies that $f(R)$ is a point. This contradicts $F\cdot R > 0$ and proves the claim.
        
        \medskip
        
        Therefore, $X\simeq\mathbb{P}^1\times Y\simeq \mathbb{P}^1\times F$ and $\varphi$ is the second projection.
        
Now, we consider when $F\cdot \ell=0$. This implies $F=\varphi^*(B)$ for some irreducible curve $B$ on $Y$, which gives a factorisation
        $f\colon X\xrightarrow{\varphi}Y\rightarrow\mathbb{P}^1$, where $Y$ is a smooth rational surface.

    \medskip
    {\bf Case ${\rm{dim}}\ Y=1$.} 
    In this case, we have $Y\simeq\pr^1$ and $\rho(X)=2$. This contradicts the fact $\rho(X)\geq 3$ by \cite[Corollary 7.3]{MR2129540}.
\end{proof}

\section{Classification result}\label{sec:classification}
The aim of this section is to prove Theorem \ref{mainnonrat}. In the whole section, we consider the following setup.
\begin{set}\label{setupnonrat}
Under Setup \ref{generalsetup}, assume that the general fibre $F$ is non-rational.
\end{set}

\begin{rem}\label{rem:notprod}
  {\rm When $F$ is non-rational, $X$ cannot be a product (i.e.\ $X\not\simeq F\times \pr^1$), since otherwise $X$ is not rationally connected.}
\end{rem}

By \cite[Proposition 1.6]{MR2129540}, we have $F=\mathbb{P}(\mathcal{E})$, where $\mathcal{E}$ is a rank-$2$ vector bundle over an elliptic curve defined by an extension
\[
0\to \mathcal{O}\to \mathcal{E}\to\mathcal{L}\to 0
\]
with $\mathcal{L}$ a line bundle of degree $0$ and either
\begin{enumerate}[label=(\roman*)]
    \item $\mathcal{L}=\mathcal{O}$ and the extension is non-split, or
    \item $\mathcal{L}$ is not torsion.
\end{enumerate}
The structure of the unique element in $|{-}K_F|$ is either
\begin{enumerate}[label=(\roman*)]
    \item $2C$, where $C$ is a smooth elliptic curve, or
    \item $C_1+C_2$, where $C_1$ and $C_2$ are smooth elliptic curves which do not meet. 
\end{enumerate}

\subsection{Structure of the fixed part}\label{subsec:fixedpart}
In this subsection, we will describe the fixed divisor of the anticanonical system and prove the following result.

\begin{prop}\label{nonratfixedpart}
In Setup \ref{setupnonrat}, $A_v=0$ and $A_h$ is not a prime divisor. We are in one of the following cases:
\begin{enumerate}[label=\normalfont(\roman*)]
    \item $-K_F=2C$, where $C$ is a smooth elliptic curve. Then $A_h=2D$, where the restriction of $f$ to $D$ induces an elliptic fibration.
    \item $-K_F=C_1+C_2$, where $C_1$ and $C_2$ are smooth elliptic curves which do not meet. Then $k=2$ and $A_h=D_1+D_2$, where for $i=1,2$, $D_i\simeq C_i\times \pr^1$ and $f|_{D_i}$ is the second projection.
\end{enumerate}
\end{prop}

We start by studying the $f$-horizontal part of the fixed divisor $A$ of $|{-}K_X|$.
\begin{lem}\label{notprime}
In Setup \ref{setupnonrat}, $A_h$ is not a prime divisor and one of the following cases occurs:
\begin{enumerate}[label=\normalfont(\roman*)]
    \item $-K_F=2C$, where $C$ is a smooth elliptic curve, then $A_h=2D$, where the restriction of $f$ to $D$ induces an elliptic fibration.
    \item $-K_F=C_1+C_2$, where $C_1$ and $C_2$ are smooth elliptic curves which do not meet, then $A_h=D_1+D_2$, where the restriction of $f$ to $D_i$ induces an elliptic fibration for $i=1,2$, and $D_1\cap D_2$ is contained in some fibres of $f$. 
\end{enumerate}
\end{lem}
\begin{proof}
Suppose by contradiction that $A_h$ is a prime divisor. By the adjunction formula, we have
\[
-K_{A_h}\sim(-K_X-A_h)|_{A_h}\sim(A_v+kF)|_{A_h}.
\]
Clearly $F|_{A_h}$ is nef and $F\cap A_h$ contains a smooth elliptic curve. Moreover, $(F|_{A_h})^2=0$ and $k\geq 2$, so we can apply Lemma \ref{surfacetwocpnts} to the surface $A_h$, which implies that $A_v|_{A_h}=0$, $-K_{A_h}\sim kF|_{A_h}$ and $A_h$ is a $\mathbb{P}^1$-bundle over a smooth elliptic curve.

On the other hand, the restriction of $f$ to $A_h$ induces a fibration on $A_h$ such that the general fibre $F|_{A_h}$ is either
\begin{itemize}
    \item $2C$, or
    \item $C_1+C_2$
\end{itemize}
with $C$, $C_1$ and $C_2$ smooth elliptic curves. The first case is impossible by the generic smoothness of the $\pr^1$-bundle $A_h$. In the second case, $-K_{A_h}\sim kF|_{A_h}$ with $k\geq 2$ contains at least $4$ elliptic curves (counted with multiplicity). Let $\ell\subset A_h$ be a fibre of the ruling. Then $-K_{A_h}\cdot \ell\geq 4$, which contradicts $-K_{A_h}\cdot\ell=2$ for a $\pr^1$-bundle over a smooth curve.
\end{proof}

\begin{lem}\label{product}
In case $(ii)$ of Lemma \ref{notprime}, we have $k=2$, $A_v=0$ and $A=D_1+D_2$, where $D_1$ and $D_2$ are disjoint with $D_i\simeq C_i\times\mathbb{P}^1$, $i=1,2$.
\end{lem}
\begin{proof}
For $i,j=1,2$ with $i\neq j$, the adjunction formula gives
$$
-K_{D_i}\sim (-K_X-D_i)|_{D_i}\sim (A_v+D_j+kF)|_{D_i}.
$$
Recall that $F|_{D_i}$ is an elliptic curve, $(F|_{D_i})^2=0$ and $k\geq 2$, then by Lemma \ref{surfacetwocpnts}, we have $A_v|_{D_i}=D_j|_{D_i}=0$, $-K_{D_i}\sim kF|_{D_i}$, and $D_i$ is a $\mathbb{P}^1$-bundle over a smooth elliptic curve. Hence $D_1\cdot D_2=0$, and thus $D_1$ and $D_2$ are disjoint. Moreover, the support of a divisor in $|{-}K_X|$ is connected in codimension one by \cite[Lemma 2.3.9]{MR1668575}. As $A_v$ does not meet $F$ and $A_v\cdot D_1=A_v\cdot D_2=0$, we obtain $A_v=0$. Thus $A=A_h$.

Since $D_i$ is a $\mathbb{P}^1$-bundle over the smooth elliptic curve $C_i$, and the restriction of $f$ to $D_i$ induces an elliptic fibration with \[h^0\big (D_i,\mathcal{O}_{D_i}(-K_{D_i})\big )=h^0\big (D_i,(f|_{D_i})^*\mathcal{O}_{\mathbb{P}^1}(k)\big )=k+1\geq 3,\] 
we deduce $D_i\simeq C_i\times \mathbb{P}^1$ by Lemma \ref{surfaceproduct}, and thus $k=2$.
\end{proof}

In the remainder of this subsection, we study case (i) of Lemma \ref{notprime} and show that $A_v=0$ in this case.
\begin{lem}\label{flopfvertical}
In Setup \ref{setupnonrat}, assume that $A_h=2D$, where the restriction of $f$ to $D$ induces an elliptic fibration. Then after performing possibly a finite sequence of flops $\psi\colon X\dashrightarrow X'$ such that $f$ factors as $f'\circ\psi$ and $f'$ gives a fibration $f'\colon X'\to \pr^1$, we obtain that $A'_h=\psi_*(A_h)$ is $f'$-relatively nef and 
\[|{-}K_{X'}|=A'_h + A'_v + |kF'|,\]
where $A'_v\coloneqq \psi_*(A_v)$ and $F'\coloneqq \psi_*(F)\simeq F$ is a general fibre of $f'$. 

Moreover, $A'_h|_{F'}=-K_{F'}$, $A'_v|_{F'}=0$ and we can write $A'_h=2D'$ with $D'=\psi_*(D)$ such that the restriction of $f'$ on $D'$ induces an elliptic fibration.
\end{lem}

\begin{proof}
Suppose that $D$ is not $f$-relatively nef. Then there exists an $f$-vertical curve $\gamma$ such that $D\cdot \gamma<0$. 
Since $X$ is smooth, for sufficiently small $\epsilon>0$, the pair $(X,\epsilon D)$ is log-canonical. 

\medskip
{\em Claim.} There exists a $(K_X+\epsilon D)$-negative extremal ray $\Gamma$ such that $D\cdot \Gamma<0$ and $F\cdot\Gamma =0$.

By the Cone Theorem, we can write
\[
\gamma=\sum_i \lambda_i\Gamma_i + R,
\]
where 
\begin{itemize}[leftmargin=*]
    \item $\lambda_i> 0$;
    \item all the $\Gamma_i$ are $(K_X+\epsilon D)$-negative extremal rays;
    \item $(K_X+ \epsilon D)\cdot R\geq 0$.
\end{itemize}

Suppose that every $(K_X+\epsilon D)$-negative extremal ray is $D$-nonnegative. Then $D\cdot\Gamma_i\geq 0$ for all $i$.
Therefore,
\[
0>D\cdot \gamma = \sum_i \lambda_i D\cdot\Gamma_i + D\cdot R\geq D\cdot R,
\]
i.e.\ $D\cdot R<0$.
Since $(K_X+\epsilon D)\cdot R\geq 0$, we have
\[
K_X\cdot R\geq -\epsilon D\cdot R >0,
\]
which contradicts the fact that $-K_X$ is nef. Hence, we may assume $D\cdot \Gamma_1<0$. Since $F$ is nef and $F\cdot \gamma =0$, we have \[
0= F\cdot \gamma = \sum_i \lambda_i F\cdot \Gamma_i + F\cdot R \geq 0
\]
and thus for all $i$, $F\cdot \Gamma_i=F\cdot R= 0$. Hence, $\Gamma_1$ is an extremal ray satifying the assumption. This proves the claim.

\medskip
Let $c_\Gamma$ be the contraction of the extremal ray $\Gamma$ and let $\ell$ be a contracted curve. Then $D\cdot \ell<0$ and $\ell$ is $f$-vertical. Since a general fibre of $f|_D\colon D\to\pr^1$ is a smooth elliptic curve which is $D$-trivial, we obtain that $\ell$ is contained in some singular fibre of $f|_D$ and the contraction $c_\Gamma$ is small. 
This implies that $K_X\cdot \ell=0$, since there is no small $K_X$-negative extremal contraction for smooth threefolds. Hence there exists a flop $\psi^+\colon X\dashrightarrow X^+$ of $c_\Gamma$ and the flopped threefold $X^+$ is smooth by \cite[Theorem 2.4]{MR986434}.
Since a flopping curve is contained in some fibre of $f$, the map $f$ factors as $f=f^+\circ\psi^+$ such that $f^+\colon X^+\to\pr^1$ is a fibration. Moreover, $F^+\coloneqq \psi^+_*(F)$ is isomorphic to $F$. Since $\psi^+$ is an isomorphism in codimension one, we have
\[
|{-}K_{X^+}|= A_h^+ + A_v^+ + |kF^+|,
\]
where $A_h^+\coloneqq \psi^+_*(A_h)$ and $A_h^+|_{F^+}=-K_{F^+}$, $A_v^+\coloneqq \psi^+_*(A_v)$ and $A_v^+|_{F^+}=0$.

By repeating the above argument and by the termination of three-dimensional flops, see \cite[Corollary 6.19]{MR1658959}, we deduce that there exists a finite sequence of flops $\psi\colon X\dashrightarrow X'$ such that $f$ factors as $f=f'\circ\psi$, where $f'\colon X'\to\pr^1$ is a fibration and $D'\coloneqq \psi_*(D)$ is $f'$-relatively nef.
\end{proof}

In order to show that $A_v=0$ in case (i) of Lemma \ref{notprime}, we will assume by contradiction that $A_v$ is non-zero and describe the geometry of the anticanonical system $|{-}K_X|$ in the following two lemmas.
\begin{lem}\label{2D}
In case $(i)$ of Lemma \ref{notprime}, assume that $A_v$ is non-zero. Then $D\simeq C\times \pr^1$ after performing possibly a finite sequence of $f$-relative $D$-flops.
Moreover, if $\ell\simeq \pr^1$ is a fibre of the first projection $pr_1\colon D\to C$ , then $A_v^3=0$ and one of the following cases occurs:
\begin{enumerate}[label=\normalfont(\roman*)]
    \item $D\cdot \ell = -1$ and thus $[\ell]$ generates a $K_X$-negative extremal ray. 
    
    In this case, we have
     $k=2, A_v|_D\sim C$ and $ D|_D\sim -C$.
    \item $D\cdot\ell = -2$ and thus $[\ell]$ generates a $K_X$-trivial extremal ray. 
    
    In this case, we have $D|_D\sim -2C$ and either
    \begin{itemize}[leftmargin=*]
        \item $k=2, A_v|_D\sim 2C$, or
        \item $k=3, A_v|_D\sim C$.
    \end{itemize}
\end{enumerate}
\end{lem}
\begin{proof}
By Lemma \ref{flopfvertical}, we may assume that $D$ is $f$-nef (by performing possibly a finite sequence of $D$-flops).

We first note that $A$ is not nef: otherwise the relative anticanonical divisor $-K_{X/\pr^1}$ is nef and thus $f\colon X\to\pr^1$ is a product by Lemma \ref{lem:loctriv}, which gives a contradiction by Remark \ref{rem:notprod}.

Since $X$ is smooth, for sufficiently small $\epsilon>0$, the pair $(X,\epsilon A)$ is log-canonical. 
By Lemma \cite[Lemma 2.5]{Xie20}, there exists a $(K_X+\epsilon A)$-negative extremal ray $\Gamma$ such that $\epsilon A\cdot \Gamma<0$. Let $c_\Gamma$ be the contraction of the extremal ray $\Gamma$ and let $\ell$ be an integral curve contracted by $c_\Gamma$. Then $A\cdot \ell<0$ and thus $\ell$ is contained in an irreducible component of $A$.
Since $A$ is $f$-nef, we obtain that $\ell$ is $f$-horizontal and thus $\ell\subset \Supp A_h$,
\begin{equation}\label{eq:hor}
    F\cdot \ell\geq 1 \text{ and } A_v\cdot \ell\geq 0.
\end{equation}

\medskip
(A) If $c_\Gamma$ is a divisorial contraction, then $D$ is the exceptional divisor of $c_\Gamma$. Since $f_h\coloneqq D\cap F$ is an integral curve with $F\cdot f_h=0$, we have $[f_h]\not\in\Gamma$. Hence $c_\Gamma$ contracts the prime divisor $D$ horizontally to a curve that we denote by $C'$.

\medskip
{\em Step A1.} In this step, we consider the case $K_X\cdot \ell<0$.
Then $c_\Gamma$ is a $K_X$-negative extremal contraction and $D$ is a ruled surface with fibre $\ell$ by the classification of Mori, see \cite[Section 3]{MR715648}. Then $D\cdot \ell =-1$ and $-K_X\cdot \ell =1$, i.e.\ $(2D+A_v+kF)\cdot \ell = 1$ with $k\geq 2$ and $A_v\neq 0$. Moreover, $\ell$ moves on the surface $D$ and thus $A_v\cdot \ell>0$. Hence, we obtain $A_v\cdot \ell = 1,F\cdot \ell =1$ and $k=2$.

\medskip
{\em Step A2.} In this step, we consider the case $K_X\cdot \ell =0$. Then $c_\Gamma$ is an extremal crepant contraction. 

We will prove in this step that $D$ is a ruled surface with fibre $\ell$.

By Proposition \ref{wilson}, $c_\Gamma|_D\colon D\to C'$ is a conic bundle whose possible singular fibre is two lines. 
Now suppose by contradiction that there exists a singular fibre $\ell_0$ of $c_\Gamma|_D\colon D\to C'$, consisting of two lines $\ell_1$ and $\ell_2$.
Then
\[
D\cdot \ell_1=D\cdot \ell_2=-1.
\]
On the other hand, since $[\ell_i]\in \Gamma$ for $i=1,2$, we have $F\cdot \ell_i\geq 1$ and $A_v\cdot \ell_i\geq 0$ by (\ref{eq:hor}). Since
\[
2D\cdot \ell_i + A_v\cdot \ell_i + kF\cdot\ell_i = -K_X\cdot \ell_i=0,
\]
we obtain $A_v\cdot \ell_i + kF\cdot \ell_i=2$ and thus $A_v\cdot \ell_i=0$ and $k=2$. Therefore $A_v\cdot\ell=0$. We have $A_v\cdot A_h=0$: otherwise $A_h\cdot A_v$ is a non-zero effective $1$-cycle which is $f$-vertical, and thus the divisor $(c_\Gamma)_*(A_v)$ contains the curve $C_0$, which contradicts the fact that $A_v\cdot \ell=0$.

Since $A_v\cdot F=0$, $A_v\cdot A_h=0$ and $-K_X\sim A_v + A_h + kF$ is nef, we obtain that $A_v$ is nef and thus $A_v=0$, which contradicts the assumption of the lemma. We conclude that $c_{\Gamma}|_D\colon D \to C'$ is a regular conic bundle and thus $D$ is a ruled surface with fibre $\ell$.

Therefore, $D\cdot \ell =-2$ and $-K_X\cdot\ell =0$, i.e.\ $(2D+A_v+kF)\cdot \ell = 0$ with $k\geq 2$ and $A_v\neq 0$. Moreover, $\ell$ moves on the surface $D$ and thus $A_v\cdot \ell>0$. Hence, we obtain $F\cdot \ell =1$ and either $A_v\cdot \ell =1$, $k=3$ or $A_v\cdot \ell =2$, $k=2$.

\medskip

{\em Step A3.} In both {\em Steps A1} and {\em A2}, the curve $C'$ is a smooth elliptic curve as $F\cdot \ell =1$. Thus $D$ is a ruled surface over a smooth elliptic curve whose fibres are $f$-horizontal. Moreover, $f|_D\colon D\to\pr^1$ induces an elliptic fibration. 

In the remainder of this step, we will show that $D\simeq C\times\ell$, where $\ell\simeq\pr^1$ and $C=F|_D$ is a smooth elliptic curve.

Denote by $Y$ the threefold obtained by the contraction $c_{\Gamma}$.
Then $D=\mathbb{P}(N_{C'/Y}^*)$. 
Let $V=N_{C'/Y}^*\otimes\mathcal{L}$ with $\mathcal{L}\in \Pic(C')$ be a normalisation of $N_{C'/Y}^*$, see \cite[Chapter V, Proposition 2.8]{MR0463157}. Let $\mu\coloneqq \deg\mathcal{L}$ and let $C_0$ be a canonical section of the tautological line bundle $\mathcal{O}_{\mathbb{P}(V)}(1)$ such that $C_0^2=-e = c_1(V)$.
Then
$-K_D\equiv 2C_0+ e\ell$,
\begin{enumerate}[label=\normalfont(\alph*)]
    \item $N_{D/X}\equiv -C_0 + \mu \ell, -K_X|_D \equiv C_0 + m\ell$ with $m\geq e$; 
    \item $N_{D/X}\equiv -2C_0 + \mu \ell, -K_X|_D \equiv m\ell$ with $m\geq 0$.
\end{enumerate}
Note that $e=0$ or $-1$ by \cite[Theorem 5]{MR242198}, since $D$ is a ruled surface over a smooth elliptic curve and admits an elliptic fibration.

By the adjunction formula, we have \[-K_D\sim (-K_X-D)|_D\sim (D+A_v+kF)|_D,\]
i.e.\ $2C_0+e\ell\equiv 2C_0+ m\ell -\mu\ell$.
Hence $\mu = m - e$ and
\begin{enumerate}[label=\normalfont(\alph*)]
\item $2C_0+e\ell \equiv -C_0 + (m-e)\ell + A_v|_D + kF|_D$,
i.e.\ $3C_0 + (2e-m)\ell \equiv A_v|_D + kF|_D$ with $m\geq e$;
\item $2C_0+e\ell \equiv -2C_0 + (m-e)\ell + A_v|_D + kF|_D$,
i.e.\ $4C_0 + (2e-m)\ell \equiv A_v|_D + kF|_D$ with $m\geq 0$.
\end{enumerate}
Since $F|_D$ and $A_v|_D$ are non-zero effective $f$-vertical $1$-cycles, we obtain in both cases $2e-m\geq 0$ and thus $e\geq 0$. Hence $e=0$ and $m=0$. Since moreover $A_v|_D$ is non-zero and $k\geq 2$, we obtain that $A_v|_D- C_0$ is effective, $F|_D\sim C_0$ and thus $C_0$ moves on the surface $D$.
Therefore, \[h^0\big (D,\ol_D(-K_D)\big ) = h^0\big (D,\ol_D(2C_0)\big )\geq 3. \] 
By Lemma \ref{surfaceproduct}, we conclude that $D$ is a product.


\medskip
(B) If $c_\Gamma$ is a small contraction, then $K_X\cdot\ell=0$ and $\ell$ is a flopping curve. Since $\ell$ is an $f$-horizontal curve contained in $A_h=2D$ and $D|_F$ is a smooth elliptic curve, $A$ has multiplicity $2$ along $\ell$ and thus $k=2$ by Lemma \ref{flopk}. The remainder of the proof is devoted to show $D\simeq C\times\pr^1$.

\medskip
\emph{Step B1.} We first prove that $-K_D\cdot \ell\leq 1$.

Assume by contradiction that $-K_D\cdot \ell\geq 2$.
Since a general fibre of $f|_D\colon D\to \pr^1$ is smooth and the curve $\ell\subset D$ is $f$-horizontal, we obtain that $\ell$ intersects the smooth locus of $D$. Hence the smooth rational curve $\ell$ deforms on the surface $D$ by \cite[Chapter 2, Theorem 1.14]{MR1440180}. This contradicts the fact that the extremal ray $\Gamma=\R_+[\ell]$ contains only a finite number of curves.

\medskip
\emph{Step B2.} We show that 
\begin{equation}\label{eq:flopcurve}
    D\cdot \ell =-1, A_v\cdot \ell =0 \text{ and } F\cdot \ell =1.
\end{equation}

By the adjunction formula, since $F\cdot \ell\geq 1$, we have 
\[
-K_D\cdot \ell = (-K_X-D)\cdot \ell = D\cdot \ell + A_v\cdot \ell + 2F\cdot\ell \geq  D\cdot \ell + A_v\cdot \ell + 2.
\]
Since $-K_D\cdot \ell\leq 1$, we have $D\cdot \ell + A_v\cdot \ell \leq -1$.
On the other hand, since $-K_X\cdot \ell =0$, we have $-D\cdot \ell = -K_D\cdot \ell \leq 1$, i.e.\ $D\cdot \ell\geq -1$. Together with $D\cdot \ell<0$ and $A_v\cdot \ell\geq 0$, we obtain $D\cdot \ell =-1$ and $A_v\cdot \ell =0$. Since $-K_X=2D+A_v+2F$ and $-K_X\cdot \ell=0$, we obtain $F\cdot \ell =1$.

\medskip
\emph{Step B3.} We prove that $-K_D$ is nef.

Assume first that $D$ is a product of $\pr^1$ and a smooth elliptic curve, then we have $-K_D\sim 2F|_D$ is nef. It remains to consider the case when $D$ is not a product.

Let $B$ be an integral curve contained in $D$. If $B$ is $f$-vertical, then $D\cdot B=0$ by Lemma \ref{verticalcurveinAh}. By the adjunction formula, we have
\[-K_D\cdot B=(-K_X-D)\cdot B= -K_X\cdot B\geq 0.\]

If $B$ is $f$-horizontal, then by the Cone Theorem we can write
\begin{equation}\label{eq:sumB}
    B=\sum_i m_i \ell_i + R,
\end{equation}
where 
\begin{itemize}[leftmargin=*]
\item  $m_i\in \mathbb{Q}^+$;
\item each $\ell_i$ is $D$-negative (and thus $(K_X+\epsilon D)$-negative) and generates a $(K_X+\epsilon D)$-negative extremal ray $\Gamma_i=\R_+[\ell_i]$ for $\epsilon>0$ sufficiently small such that $(X,\epsilon D)$ is a log-canonical pair;
\item $R$ is $D$-nonnegative.
\end{itemize}
Note that each $\ell_i$ in (\ref{eq:sumB}) is $K_X$-trivial: otherwise $\ell_i$ generates a $K_X$-negative extremal ray $\Gamma_i$ with $D\cdot \ell_i<0$, thus $D$ is the exceptional divisor of the contraction associated to $\Gamma_i$ and $\ell_i$ is $f$-horizontal. Thus by {\em Step A1}, we obtain $D\simeq C\times \ell_i$ with $C=F|_D$ a smooth elliptic curve. This contradicts the assumption that there exists a flopping curve $\ell$ contained in $D$.

Note also that {\em Step B1} and {\em Step B2} hold for all $\ell_i$, which satisfies that $\Gamma_i=\R_+[\ell_i]$ is a $K_X$-trivial extremal ray and $D\cdot \ell_i<0$. Hence by (\ref{eq:flopcurve}),
we obtain 
\begin{equation}\label{eq:aux}
     D\cdot \ell_i = -1, F\cdot \ell_i = 1.
\end{equation}
Thus,
\begin{align*}
    -K_D\cdot B &= -K_X\cdot B - D\cdot B \\
    &= D\cdot B + A_v\cdot B + 2F\cdot B \\
    &= \sum_i m_i D\cdot \ell_i + D\cdot R + A_v\cdot B + 2\sum_i m_i F\cdot \ell_i + 2F\cdot R & \text{ by (\ref{eq:sumB})}\\
    &= \sum_i m_i + D\cdot R + A_v\cdot B +  2F\cdot R & \text{ by (\ref{eq:aux})}\\
    &\geq 0.
\end{align*}

\medskip
\emph{Step B4.} In this step we conclude that $D$ is a product of $\pr^1$ and a smooth elliptic curve.

Assume that $D$ is not a product. Then there exists an $f$-horizontal flopping curve $\ell_0$ contained in $D$, thus $D$ is an elliptic fibration over $\pr^1$ with a section and $-K_D$ is nef.

Let $\nu\colon\widehat{D}\to D$ be the normalisation of the surface $D$ and let $\mu\colon \widetilde{D}\to\widehat{D}$ be the minimal resolution of the surface $\widehat{D}$. Denote by $\pi\coloneqq\nu\circ\mu$ the composition map. Then
\[
-K_{\widetilde{D}}= \pi^*(-K_D)+ E_0,
\]
where $E_0$ is an effective Weil divisor. Moreover, the divisor $E_0$ is $f$-vertical since a general fibre of $f|_D\colon D\to \pr^1$ is smooth.

Let $h\colon D_{\min}\to\pr^1$ be a relative minimal model of the induced elliptic fibration $\widetilde{D}\to \pr^1$, i.e.\ we take the successive blow-down $g\colon \widetilde{D}\to D_{\min}$ of $(-1)$-curves contained in fibres. Then
\[
-K_{\widetilde{D}}+E'=g^*(-K_{D_{\min}}),
\]
where $E'$ is an effective divisor. Since $f|_D\colon D\to\pr^1$ has a section, $D_{\min}$ is a smooth surface with relatively minimal elliptic fibration over $\pr^1$ having a section that we denote by $\Theta_{\min}$. By \cite[8.3]{MR2732092}, we have
\[
-K_{D_{\min}}\sim h^*\Op\big ( 2-\chi(\ol_{D_{\min}})\big ).
\]
Moreover, since $-K_D$ is nef and non-trivial, together with \cite[8.3]{MR2732092}, we are in one of the following cases:
\begin{itemize}[leftmargin=*]
\item $\chi(\ol_{D_{\min}})=0$ and thus $D_{\min}$ is a product, or
\item $\chi(\ol_{D_{\min}})=1$ and thus $D_{\min}$ is a rational surface with $-K_{D_{\min}}\sim h^*\Op(1).$
\end{itemize}

In the first case, we have 
\[
g^*\big ( h^* \Op(2) \big ) \sim -K_{\widetilde{D}}+ E' \sim \pi^*(-K_D)+ E_0 + E'.
\]
If $\pi^*(-K_D)\sim g^*\big ( h^* \Op(2) \big )$, then $E_0=0$ and $E'=0$. Thus $D$ is a rational normal surface with at worst rational singularities and its minimal resolution $\widetilde{D}$ is a product.
We conclude that $D$ is also a product, which contradicts our assumption. 
If $\pi^*(-K_D)\sim g^*\big ( h^* \Op(1) \big )$, then $E_0 + E' \sim g^*\big ( h^* \Op(1) \big )$ is a fibre of the elliptic fibration $\widetilde{D}\to\pr^1$ and $E_0 + E'$ is obtained from a smooth elliptic curve $B$ by successively blowing up points. In particular, $E_0 + E'$ is a reduced tree of smooth curves so that after contracting any irreducible component, we still obtain a reduced tree of smooth curves.
Since moreover $E'$ consists of smooth rational curves, we have $B\subset\Supp E_0$ and thus $\pi_*(E'+E_0)$ is a reduced tree of smooth rational curve. This contradicts the fact that all fibre of $f|_D$ has arithmetic genus one.

In the second case, we obtain that
\[
g^*\big ( h^* \Op(1) \big ) \sim -K_{\widetilde{D}}+ E' \sim \pi^*(-K_D)+ E_0 + E' 
\]
is a fibre of the elliptic fibration $\widetilde{D}\to\pr^1$. Note that $g^*\big ( h^* \Op(1) \big ) $ is not divisible in $\Pic(\widetilde{D})$. This is because the elliptic fibration $\widetilde{D}\to\pr^1$ has a section that we denote by $\widetilde{\Theta}$ and \[g^*\big ( h^* \Op(1) \big ) \cdot \widetilde{\Theta} = h^* \Op(1)\cdot \Theta_{\min} = 1.\]
Now since $\pi^*(-K_D)$ is non-zero, effective and nef, and $E_0 + E'$ is an effective $f$-vertical divisor, we obtain
\[
E_0=0, E'=0 \text{ and }-K_D\sim F|_D.
\]
Hence $D$ is a rational normal surface with at worst rational singularities, and it has a relatively minimal elliptic fibration induced by $f|_D\colon D\to \pr^1$. By the adjunction formula,
\[
F|_D\sim -K_D \sim D|_D + A_v|_D + 2F|_D,
\]
and we obtain
\[
(D|_D)^2 = (-F|_D-A_v|_D)^2 = (A_v|_D)^2,
\]
as $F^2=F\cdot A_v=0$. On the other hand, since $D|_D$ is an $f$-vertical $1$-cycle contained in $D$ and $D$ is $f$-relatively nef by assumption, we have $(D|_D)^2=0$ by Lemma \ref{verticalcurveinAh}. Hence $(A_v|_D)^2=0$ and thus by 
\cite[Corollary 2.6]{MR1805816}, $A_v|_D\sim mF|_D$ for some $m>0$. Therefore, $A_v\cdot \ell_0 = m>0$, which contradicts the fact $A_v\cdot \ell_0 = 0$.
\end{proof}

\begin{lem}\label{Avstructure}
In the setting of Lemma \ref{2D}, let $A_0$ be a connected component of $A_v$ contained in some fibre $F_0$ of $f$ and let $A_1$ be an irreducible component of $A_0$ such that $A_1\cap D = F_0\cap D \eqqcolon C_0$, where $C_0$ is a smooth elliptic curve. Then $A_0 = A_1$ is a $\pr^1$-bundle over a smooth elliptic curve such that
\[
-K_{A_1}\sim 2C_0
\]
is nef.
\end{lem}

\begin{proof}
We first note that $D\cap F_0$ is reduced: by Lemma \ref{2D} and using the same notation, we have $F_0|_D\sim m C$ with $m\in \N^*$ and $F_0|_D\cdot \ell = F|_D\cdot \ell =1$, i.e.\ $m C\cdot \ell =1$ on the smooth surface $D$ and thus $m=1$.

By the adjunction formula,
\[
-K_{A_1}\sim (-K_X-A_1)|_{A_1} \sim (2D+A_0-A_1)|_{A_1}=2C_0 + (A_0-A_1)|_{A_1}.
\]
Since $D\cap A_0=C_0$ and $A_1|_D=C_0$, $A_1-A_0$ does not contain $A_0$ as component. Thus the $1$-cycle $(A_0-A_1)|_{A_1}$ on $A_1$ is effective and does not meet $C_0$.

On the surface $A_1$, we have
\[
C_0^2 = (D|_{A_1})^2 = D^2\cdot A_1 = 0,
\]
since $D^2\equiv -C$ or $-2C$ by Lemma \ref{2D}. Therefore, by Lemma \ref{surfacetwocpnts}, $(A_0-A_1)|_{A_1}=0$ and the surface $A_1$ is a $\pr^1$-bundle over a smooth elliptic curve. Moreover, $A_0-A_1=0$ as $A_0$ is connected.
\end{proof}

Now, it remains to exclude the two cases described in Lemma \ref{2D}. We start by excluding the first case.
\begin{lem}\label{case1}
Case (i) of Lemma \ref{2D} cannot happen.
\end{lem}
\begin{proof}
In case (i) of Lemma \ref{2D}, $A_v|_D\sim C$. Thus $A_v$ has a unique connected component and it is a $\pr^1$-bundle over a smooth elliptic curve by Lemma \ref{Avstructure}.

Suppose that $D$ is horizontally contracted to a smooth elliptic curve by a $K_X$-negative extremal contraction $\varphi\colon X\to Y$. Then by Lemma \ref{contractfixedcpnt}, we obtain that $-K_Y$ is nef, and \[\kappa(Y, -K_Y)=\kappa(X, -K_X)=1, \quad \nu(Y, -K_Y)=\nu(X, -K_X)=2.\]

As $\varphi$ contracts the curves meeting $F$ (resp.\ $A_v$) transversally at one point, $G\coloneqq\varphi(F)\simeq F$, and $A'_v\coloneqq \varphi (A_v)\simeq A_v$ is a $\pr^1$-bundle over a smooth elliptic curve by Lemma \ref{Avstructure}. Moreover, we have \[-K_Y\sim A'_v + 2G,\] and two general members in $|G|$ (resp.\ $A'_v$ and a general member in $|G|$) intersect along the smooth elliptic curve $C'\coloneqq \varphi_*(D)$.

Note that $|{-}K_Y|$ must have a non-zero fixed divisor and thus $A'_v$ is the fixed divisor: otherwise $A'_v$ is divisible by two in $\Pic(Y)$ by \cite[Theorem 1.1]{Xie20}, and thus $A'_v|_G$ is divisible by two in $\Pic(G)$, which contradicts the fact that $A'_v|_G\sim C'$ is a section of the $\pr^1$-bundle $G$.

Therefore,
\[|{-}K_Y|=A'_v+|2G|.\] 
Since $G\cdot C'=C'^2=0,$ where the last intersection number is computed on the surface $G$, we deduce that $G$ is nef but $G^2\neq 0$. This contradicts Theorem \ref{introfixed}.
\end{proof}

Finally, we will exclude the second case described in Lemma \ref{2D} and conclude our proof by contradiction that $A_v$ is indeed zero.
\begin{lem}\label{case2birat}
Consider case (ii) of Lemma \ref{2D}. Let $\varphi\colon X\to Y$ be a divisorial $K_X$-negative extremal contraction and let $E$ be the exceptional divisor. Then $E$ is contained in some special fibre of $f$ which contains an irreducible component of $A_v$ and $E$ is not a component of $A_v$. 

Moreover, $E$ is contracted by $\varphi$ to a smooth curve of positive genus. Thus $Y$ is smooth with $-K_Y$ nef, the fibration $f$ factors as $f=f'\circ \varphi$ such that $f'\colon Y\to \mathbb{P}^1$ gives $Y$ the fibration structure and $\varphi(F)\simeq F$, $\varphi(D)\simeq D$, $\varphi(A_v)\simeq A_v$.
\end{lem}
\begin{proof}
Let $\Gamma$ be the $K_X$-negative extremal ray corresponding to the contraction $\varphi$.
Let $\gamma$ be a rational curve such that $[\gamma]$ generates $\Gamma$ and $-K_X\cdot \gamma=l(\Gamma).$

\medskip
{\bf Case $A\cdot \gamma=0$.} 
In this case, we have $F\cdot \gamma=1$ and $-K_X\cdot \gamma=A\cdot \gamma+kF\cdot \gamma=k=2$. Hence $\varphi$ is the blow-up of $Y$ at a smooth point, with exceptional divisor $E\simeq\mathbb{P}^2$. As $E$ is not fibred, it is contained in a fibre of $f$ and thus $F\cdot E=0$. This contradicts the fact that $F\cdot \gamma=1$.

\medskip

{\bf Case $A\cdot \gamma<0$.} 
Since the contraction is divisorial and $F\cdot \gamma>0$, $E$ is an irreducible component of $A_h$, i.e.\ $E=D$. Since $D\simeq C\times\pr^1$, where $C$ is a smooth elliptic curve, is the exceptional divisor of an $K_X$-trivial extremal contraction by assumption of the lemma, this contradicts the fact that $E$ is the exceptional divisor of a $K_X$-negative contraction.

\medskip

{\bf Case $A\cdot \gamma>0$.} In this case, we have $F\cdot \gamma = 0$ since otherwise $-K_X\cdot \gamma>2$, which contradicts the classification of Mori.

\medskip
{\em Step 1.} We first note that $E$ is not an irreducible component of $A_h$ by the same argument as in the previous case.

In the remainder of this step, we will show that $E$ is not an irreducible component of $A_v$.

Suppose by contradiction that $E$ is an irreducible component of $A_v$ contained in $F_0$. Then $E$ is a $\pr^1$-bundle over a smooth elliptic curve by Lemma \ref{Avstructure} and thus $E$ contracted to a smooth elliptic curve by $\varphi$ and $-K_Y$ is nef with $\kappa(Y, -K_Y)=\kappa(X, -K_X)=1$, $\nu(Y, -K_Y)=\nu(-X, K_X)=2$ by Lemma \ref{contractfixedcpnt}.
    
     We have $G\coloneqq\varphi(F)\simeq F$, $D'\coloneqq\varphi(D)\simeq D$, and the fibration $f$ factors as $f=f'\circ \varphi$ such that $f'\colon Y\to \mathbb{P}^1$ gives $Y$ the fibration structure. Since $E\subset A_v$ intersects $D$ along a smooth elliptic fibre of $D$ by Lemma \ref{Avstructure}, $D'$ is now the exceptional divisor of a $K_Y$-negative extremal contraction which contracts $D'$ horizontally to a smooth elliptic curve. Hence
     $D'$ is not a component of the fixed divisor of $|{-}K_Y|$ by Lemma \ref{case1}.
     
     If $k=3$, then $|{-}K_Y|=|2D' + 3G|$ has no fixed divisor, as $D'$ is not a fixed divisor of $|{-}K_Y|$ and $h^0\big ( Y,\ol_Y(G)\big ) = h^0\big ( \pr^1,\ol_{\pr^1}(1)\big ) =2$. Since $G$ is not divisible by two in $\Pic(Y)$, this case cannot happen by \cite[Theorem 1.1]{Xie20}.
     
     If $k=2$, then $-K_Y\sim A'_v+2D'+2G$, where $A'_v$ is the image of the other irreducible component of $A_v$, which is a $\pr^1$-bundle over a smooth elliptic curve by Lemma \ref{Avstructure}. Note that $|{-}K_Y|$ must have a non-zero fixed divisor and thus $A'_v$ is the fixed divisor: otherwise $A'_v$ is divisible by $2$ in $\Pic(Y)$ by \cite[Theorem 1.1]{Xie20}; but $-K_{D'}\sim 2G|_{D'}$ as $D'\simeq D$ is a product, and by the adjunction formula,
     \[ -K_{D'}\sim (-K_Y-D')|_{D'}\sim (A'_v + D'+ 2G)|_{D'}\sim A'_v|_{D'} + G|_{D'},\]
     and thus we obtain that $G|_{D'}\sim A'_v|_{D'}$ is divisible by two in $\Pic(D')$, which contradicts the fact that $G|_{D'}$ is a fibre of the product $D'$.
     Hence $|{-}K_Y|=A'_v + |2D' + 2G|$.
     
     Since $D'+G$ is nef but \[(D'+G)^2 = D'^2 + 2D'\cdot G = D'\cdot G\] is a non-zero effective $1$-cycle, this cannot happen by Theorem \ref{introfixed}.

\medskip
{\em Step 2.} Denote by $A_1$ an irreducible component of $A_v$.
Since $A_1$ (resp.\ $D$ and resp.\ $F$) is a $\pr^1$-bundle over a smooth elliptic curve, we deduce that $E\cap A_1$ (resp.\ $E\cap D$ and resp.\ $E\cap F$) does not contain $\gamma$ or any of its deformations: otherwise $\gamma$ moves on the surface $A_1$ (resp.\ $D$ and resp.\ $F$), which gives a contradiction. Hence, we obtain the following conclusions:
\begin{itemize}[leftmargin=*]
\item $E$ is contained in some special fibre $F_0$ of $f$, thus $\varphi$ is an isomorphism outside $F_0$ and $\varphi(F)\simeq F$;
\item $E$ is contracted to a smooth curve and thus $A\cdot \gamma =1$. This is because $E$ must meet $A=2D+A_v$, but $E\cap D$ (resp.\ $E\cap A_v$) does not contain $\gamma$ or any of its deformations.
\end{itemize}

Therefore, $D\cdot \gamma\geq 0$ and $A_v\cdot \gamma \geq 0$. Since $A\cdot \gamma= (2D+A_v)\cdot \gamma = 1$, we obtain $D\cdot \gamma =0$ and $A_v\cdot \gamma = 1$. From now on, we consider $A_1$ as the irreducible component of $A_v$ contained in $F_0$. Since $\gamma$ meets $A_1$ transversally at one point, we have $\varphi(A_1)\simeq A_1$. As $E$ does not meet other connected component of $A_v$, we have $\varphi(A_v)\simeq A_v$.

Since $D\cdot \gamma =0$ and $D$ does not contain $\gamma$ or any of its deformations, we deduce that $E\subset F_0$ is disjoint from $D$. Thus $\varphi(D)\simeq D$ and the restriction $E|_{A_1}$ is disjoint from $C_0\coloneqq A_1\cap D$. Moreover, $A_1$ is a $\pr^1$-bundle over a smooth elliptic curve, thus $E$ is contracted to a smooth curve of genus at least one. Hence by \cite[Proposition 3.3]{MR1255695}, $-K_Y$ is nef.
\end{proof}

\begin{cor}\label{Aviszero}
In case (i) of Lemma \ref{notprime}, we have $A_v=0$.
\end{cor}

\begin{proof}
Suppose that $A_v$ is non-zero. It remains to exclude case (ii) of Lemma \ref{2D}.

In the setting of case (ii) of Lemma \ref{2D}, after finitely many steps of divisorial contractions, we may assume, by Lemma \ref{case2birat}, that $X$ satisfies the following:
\begin{itemize}[leftmargin=*]
\item $X$ is smooth and rationally connected, and it has a fibration structure $f\colon X\to\pr^1$ with general fibre $F$, where $F$ is a $\pr^1$-bundle over a smooth elliptic curve such that $-K_F$ is nef and not semi-ample;
\item $-K_X$ is nef and $-K_X\sim 2D+A_v+kF$, where $k\geq 2$, $f|_D\colon D\simeq C\times \pr^1\to \pr^1$ is the second projection and $C$ is a smooth elliptic curve, $A_v$ is contained in some fibres of $f$ and every connected component of $A_v$ is a $\pr^1$-bundle over a smooth elliptic curve.
\item If $\varphi\colon X\to Y$ is a $K_X$-negative extremal contraction, then $\varphi$ is non-birational.
\end{itemize}
Let $\Gamma$ be the $K_X$-negative extremal ray corresponding to the contraction $\varphi$.
Let $\gamma$ be a rational curve such that $[\gamma]$ generates $\Gamma$ and $-K_X\cdot \gamma=l(\Gamma).$

Note that $\varphi$ is not a del Pezzo fibration. Otherwise we have $Y\simeq \mathbb{P}^1$, $-K_X\cdot \ell \in \{1,2,3\}$, and every fibre of $\varphi$ is integral, with general fibre isomorphic to a smooth del Pezzo surface. Since $[\gamma]$ is a movable class, we have $D\cdot \gamma\geq 0$ with equality if and only if $D$ is a fibre of $\varphi$.
    Since $-K_D$ is not ample, $D$ is not a fibre of $\varphi$ and one has $D\cdot \gamma > 0$. If $-K_X\cdot \gamma\in \{1,2\}$, this implies $F\cdot \gamma=0$ and thus $F=\varphi^*(p)$ with $p\in Y$. Therefore, $f$ and $\varphi$ coincide.
    If $-K_X\cdot \ell=3$, then $\varphi$ is a $\mathbb{P}^2$-bundle.
    As $\mathbb{P}^2$ is not fibred, $F$ restricted to a $\mathbb{P}^2$ is trivial. Hence again, the two fibrations $\varphi$ and $f$ coincide. This contradicts the fact that $-K_F$ is not ample.

Since $X\not\simeq F\times \pr^1$, by Proposition \ref{generalnonbiratcontr}, $\varphi\colon X\to Y$ is a conic bundle such that $f$ factors as $f=f'\circ \varphi$ and $Y$ a smooth rational surface. Then any fibre of $\varphi$ is an $f$-vertical curve and we have $F\cdot \gamma =0$, $D\cdot \gamma >0$. Moreover, since $-K_X\cdot \gamma = (2D+A_v+kF)\cdot \gamma \in \{1,2\}$ and $A_v\cdot \gamma\geq 0$, we have $A_v\cdot \gamma =0$ and $D\cdot \gamma =1$. Hence $\varphi$ induces a birational morphism from $D$ to $Y$.
We obtain a contradiction, since $q(D)=1$ and $q(Y)=0$.
\end{proof}
\begin{proof}[Proof of Proposition \ref{nonratfixedpart}]
    It follows from Lemma \ref{notprime}, Lemma \ref{product} and Corollary \ref{Aviszero}.
\end{proof}

\subsection{Running the Minimal Model Program}\label{subsec:mmp}
In order to achieve the classification in Theorem \ref{mainnonrat}, we will start by running the Minimal Model Program. We first consider a birational contraction.
\begin{prop}\label{nonratbirat}
In Setup \ref{setupnonrat}, assume that there exists a birational $K_X$-negative extremal contraction $\varphi\colon X\to Y$. Then $A=D_1+D_2$ with $D_1$, $D_2$ disjoint and $\varphi\colon X\to Y$ factors as $f=f'\circ\varphi$. Furthermore, $Y$ satisfies again Setup \ref{setupnonrat} with $|{-}K_Y|=D'_1+D'_2+|2G|$, $D'_1\simeq D_1$, $D'_2\simeq D_2$, $G\simeq F$, and $\varphi$ is the blow-up of $Y$ along a smooth elliptic curve in some fibre of $f'|_{D'_i}$, for $i\in\{1,2\}$.
\end{prop}
\begin{proof}
We use the same notation as in Proposition \ref{nonratfixedpart}.
Let $\Gamma$ be the $K_X$-negative extremal ray corresponding to the contraction $\varphi$.
Let $\ell$ be a rational curve such that $[\ell]$ generates $\Gamma$ and $-K_X\cdot \ell=l(\Gamma).$ Let $E$ be the exceptional divisor of $\varphi$.
By Proposition \ref{generalbiratcontr}, $\varphi$ contracts $E$ to a smooth curve, and one of the following cases occurs.

\medskip
\noindent
(i) The divisor $E$ is an irreducible component of $A$ and $\varphi$ contracts $E$ horizontally to a smooth curve.
\begin{itemize}[leftmargin=*]
\item If $A=2D$, then $E=D$ is a $\mathbb{P}^1$-bundle over a smooth elliptic curve and $\varphi$ contracts $E$ to a smooth elliptic curve. This implies that $Y$ is smooth with $-K_Y$ nef. In this case, we have $D\cdot \ell=-1$, $-K_X\cdot \ell=2D\cdot \ell+kF\cdot \ell=1$, and thus $F\cdot \ell=1$, $k=3$.
    
As we contract the curves meeting $F$ transversally, we conclude that $G\coloneqq\varphi(F)\simeq F$. Since \[-K_Y=\varphi_*(-K_X)=\varphi_*(A+3F)=3\varphi_*(F)=3G,\]
we see that $|{-}K_Y|=|3G|$ has no fixed divisor. This contradicts the fact that $-K_Y$ is divisible by two in $\Pic(Y)$ by \cite[Theorem 1.1]{Xie20}.
\item If $A=D_1+D_2$, then $E=D_i$ with $i=1$ or 2 and $\varphi$ contracts $E$ to a smooth elliptic curve. Thus $Y$ is smooth and $-K_Y$ is nef. In this case, we have $D_i\cdot \ell=-1$ and $F\cdot \ell=1$.
    
    As we contract the curves meeting $F$ transversally, we conclude that $G\coloneqq \varphi(F)\simeq F$ and two general members in $|G|$ meet along a smooth elliptic curve $C'\coloneqq \varphi_*(D_i)$.
    
    We have \[-K_Y=\varphi_*(-K_X)=\varphi_*(A+2F)=D'+2G,\] where $D'\coloneqq\varphi(D_j)$ with $j\neq i$.
    
    Note that $|{-}K_Y|$ has a non-zero fixed part and thus $D'$ is the fixed part: otherwise $D'$ is divisible by two in $\Pic(Y)$ and thus $D'|_G$ is divisible by two in $\Pic(G)$. Since $D'|_G$ is a smooth elliptic curve, it is a section of the $\mathbb{P}^1$-bundle $G$. This contradicts the fact that $D'|_G$ is divisible by two.
Therefore, \[|{-}K_Y|=D'+|2G|.\] 
Now on the surface $G$, one has $C'^2=0$. We conclude that $G^2= C'$ is non-zero, and $G$ is nef as $G\cdot C'=0$. This cannot happen by Theorem \ref{introfixed}.
\end{itemize}

\medskip
\noindent
(ii) The fibration $f$ factors as $f=f'\circ\varphi$, which gives a fibration $f'\colon Y\to\pr^1$.
Since there is no $(-1)$-curve in $F$, we deduce that $E$ is contained in some fibre $F_0$ of $f$. We have $F\cdot \ell=0$ and $A\cdot \ell=1$.

\begin{itemize}[leftmargin=*]
\item If $A=2D$, then $D\cdot \ell=\frac{1}{2}$, which contradicts the fact that $D$ is a Cartier divisor.
\item If $A=D_1+D_2$, then $A\cap E = C_i$, where $i=1$ or $2$. Hence $E$ is contracted to a smooth elliptic curve by $\varphi$ and thus $-K_Y$ is nef. We have $G\coloneqq \varphi(F)\simeq F$, $D'_1\coloneqq \varphi(D_1)\simeq D_1$ and $D'_2\coloneqq \varphi(D_2)\simeq D_2$.
In this case, $Y$ satisfies again Setup \ref{setupnonrat}. Indeed, $-K_Y$ is not semi-ample, since otherwise $-K_{G}$ is semi-ample. The linear system $|{-}K_Y|$ has a non-zero fixed divisor and thus $D'_1+D'_2$ is the fixed divisor, since otherwise $D'_1+D'_2$ is divisible by two in $\Pic(Y)$ and thus $-K_G$ is divisible by two in $\Pic(G)$, which contradicts the fact $-K_G=C_1+C_2$ with $C_1$ not linearly equivalent to $C_2$.\qedhere
\end{itemize}
\end{proof}

Now we consider a contraction of fibre type.
\begin{prop}\label{nonratnonbirat}
In Setup \ref{setupnonrat}, assume that there exists a non-birational $K_X$-negative extremal contraction $\varphi\colon X\to Y$. Then $|{-}K_X|=2D+|2F|$ and $\varphi\colon X\to Y$ is a $\pr^1$-bundle. Moreover, $Y$ is isomorphic to $\pr^2$ blown up in $9$ points such that $-K_Y$ is nef and base-point-free (thus induces an elliptic fibration $\pi\colon Y\to\pr^1$), $D=\pr\big ( \mathcal{O}_Y(K_Y)\big ) $, and $X=\pr(\mathcal{V})$, where $\mathcal{V}$ is a rank-two vector bundle which is a non-split extension
\[
0\to\mathcal{O}_Y\to\mathcal{V}\to\mathcal{O}_Y(K_Y)\to 0.
\]
Furthermore, $f$ factors as $X\overset{\varphi}\rightarrow Y\overset{\pi}\rightarrow\pr^1$.
\end{prop}
\begin{proof}
Let $\Gamma$ be the $K_X$-negative extremal ray corresponding to the contraction $\varphi$.
Let $\ell$ be a rational curve such that $[\ell]$ generates $\Gamma$ and that $-K_X\cdot \ell=l(\Gamma).$

 
 


By Proposition \ref{generalnonbiratcontr}, $\varphi\colon X\to Y$ is a conic bundle and $Y$ is a smooth rational surface. Moreover, since $F$ is not rational, $X$ is not a product and $f$ factors as $f\colon X\xrightarrow{\varphi}Y\xrightarrow{\pi}\mathbb{P}^1$.
We have $F\cdot \ell=0$ and thus $F=\varphi^*(R)$ with $R$ an irreducible curve on $Y$. On the other hand, as $F=\mathbb{P}(\mathcal{E})$ where $\mathcal{E}$ is a rank-$2$ vector bundle over a smooth elliptic curve, we deduce that $R$ is a smooth elliptic curve and that the fibration $\varphi|_F$ coincides with the $\mathbb{P}^1$-bundle structure $\mathbb{P}(\mathcal{E})\to R$ on $F$. Let $\Delta$ be the discriminant locus of the conic bundle $\varphi$. Then $\Delta$ is contained in some special fibres of $\pi\colon Y\to\pr^1$. As $\varphi$ is an extremal contraction, by \cite[page 83, Remark]{MR715647}, every non singular rational curve in $\Delta$ must meet the other components of $\Delta$ in at least two points. This implies that $\Delta$ is empty.
    Therefore, $\varphi\colon X\to Y$ is a $\mathbb{P}^1$-bundle, and $-K_X\cdot \ell=A\cdot \ell=2$. We can write $X\simeq \mathbb{P}(\mathcal{V})$, where $\mathcal{V}$ is a rank-$2$ vector bundle over $Y$, and we have $\mathcal{V}|_R\simeq \mathcal{E}$.
    \begin{itemize}[leftmargin=*]
    \item If $A=2D$, then $D\cdot \ell=1$. Since $D$ is a rational section, we have an extension
    \[
    0\to\mathcal{O}_Y\to \mathcal{V}\to\mathcal{I}_Z\otimes \det\mathcal{V}\to 0,
    \]
    where $\mathcal{I}_Z$ is the ideal sheaf of a length-$c_2(\mathcal{V})$ subscheme $Z$ on $Y$ and $D=\mathbb{P}(\mathcal{I}_Z\otimes \det\mathcal{V})$.
    We have that $-K_Y$ is nef by \cite[Proposition 3.1]{MR1255695}, and that $\pi\colon Y\to\mathbb{P}^1$ induces an elliptic fibration on $Y$. Hence, $Y$ is isomorphic to $\mathbb{P}^2$ blown up at $9$ points such that $-K_Y$ is nef and semi-ample (with some multiple of $-K_Y$ defining the elliptic fibration $\pi$), and thus $-K_Y\sim \alpha R$, where $R$ is a general elliptic fibre of $\pi$ and $\alpha\leq 1$. Hence, $(-K_Y)^2=0.$ Now since $D$ is isomorphic to $\Bl_Z(Y)$ and $(-K_D)^2=(D|_D+kF|_D)^2=0$ as $A^3=A^2\cdot F=0$, we deduce that $Z=\emptyset$, $c_2(\mathcal{V})=0$, and $D\simeq Y$.
    Since
    \[
    -K_X\sim \varphi^*\big ( -K_Y-c_1(\mathcal{V})\big ) +2D,
    \]
    and $-K_X\sim 2D+kF$, we deduce that $\varphi^*\big ( c_1(\mathcal{V})\big ) \sim -(k-\alpha)F$. By the Grothendieck relation, one has $D^2\sim D\cdot\varphi^*\big( c_1(\mathcal{V})\big) \sim -(k-\alpha)D\cdot F.$ Denote by $e$ the smooth elliptic curve $D\cap F$. Then,
    \[
    (-K_X)|_D\sim (2D+kF)|_D\sim (2\alpha - k)e.
    \]
    Since $-K_X$ (and thus $-K_X|_D$) is nef, and $\alpha\leq 1$, $k\geq 2$, we deduce $\alpha = 1$ and $k=2$. Therefore, $-K_X \sim 2D+2F$ and $\mathcal{V}$ is an extension
    \[
    0\to\mathcal{O}_Y\to \mathcal{V}\to \mathcal{O}_Y(K_Y)\to 0.
    \]
    \item If $A=D_1+D_2$, then $D_1\cdot \ell=D_2\cdot \ell=1$ and $D_1, D_2$ are birational to $Y$. We obtain a contradiction, since $q(Y)=0$ and $q(D_1)=q(D_2)=1$.\qedhere
\end{itemize}
\end{proof}

Finally, we are ready to prove Theorem \ref{mainnonrat}.
\begin{proof}[Proof of Theorem \ref{mainnonrat}]
    For part (A) of the theorem, it remains to show that $\mathcal{V}$ is indecomposable. The other statements follow from Propositions \ref{nonratbirat} and \ref{nonratnonbirat}.

We first notice that
\[
\Ext^1(\mathcal{O}_Y(K_Y),\mathcal{O}_Y)\simeq H^1(Y,\mathcal{O}_Y(-K_Y))=\mathbb{C},
\]
where the last equality follows from the Riemann-Roch formula.

Now suppose by contradiction that $\mathcal{V}=\mathcal{O}_Y\oplus\mathcal{O}_Y(K_Y)$. Then the quotient
\[
\mathcal{V}\to\mathcal{O}_Y\to 0
\]
gives a section $D'$ of $\varphi\colon X=\pr(\mathcal{V})\to Y$ such that $D'\in|D-\varphi^*(K_Y)|=|D+F|$. Therefore, $-K_X\sim 2D'$ and $D'\neq D$, which contradicts the fact that $2D$ is the fixed divisor of $|{-}K_X|$.

\medskip

Now we prove part (B) of the theorem.

Let $Y$ be $\pr^2$ blown up at $9$ points such that $-K_Y$ is nef, base-point-free and thus defines an elliptic fibration $\pi\colon Y\to\pr^1$. Let $R$ be a general fibre of $\pi$. Then $-K_Y\sim R$ and $F= \varphi^*(R)$. Since $D$ is a tautological divisor of $\pr(\mathcal{V})=X$, we have
\[
-K_X\sim 2D+\varphi^*(-K_Y-\det(\mathcal{V}))\sim 2D+2F.
\]

\medskip

{\em Step 1.} We show that $F=\pr(\mathcal{E})$, where $\mathcal{E}$ is a rank-$2$ vector bundle over $R$, which is a non-split extension
\begin{equation}\label{extE}
0\to\mathcal{O}_R\to\mathcal{E}\to\mathcal{O}_R\to 0.
\end{equation}

Indeed, as $F=\varphi^*(R)$, we have $F=\pr(\mathcal{E})$ with $\mathcal{E}\simeq\mathcal{V}|_R$. Restricting the short exact sequence (\ref{extennonrat}) to $R$, we obtain
\[
0\to\mathcal{O}_R\to\mathcal{V}|_R\to\mathcal{O}_R\to 0,
\]
as $\mathcal{O}_R(K_Y)\simeq\mathcal{O}_R$. 

Let $s$ be a non-zero element in $\Ext\big (\mathcal{O}_Y(K_Y),\ol_Y\big )\simeq H^1\big (Y,\ol_Y(-K_Y)\big )\simeq H^1\big (Y,\pi^*\Op(1)\big )$.
Since $H^1\big (\pr^1,\pi_*(\pi^*\Op(1))\big )=0$,
\[
H^1\big (Y,\pi^*\Op(1)\big )\simeq H^0\big (Y,R^1\pi_*(\pi^*\Op(1))\big)
\]
by the Leray spectral sequence. As $\pi^*\Op(1)\simeq \omega_{Y/\pr^1}$, one has
\[
R^1\pi_*\big (\pi^*\Op(1)\big )\simeq  R^1\pi_*\omega_{Y/\pr^1}\simeq \Op
\]
by \cite[Proposition 7.6]{MR825838}. Hence, $\Ext\big (\mathcal{O}_Y(K_Y),\ol_Y\big )\simeq H^0\big (\pr^1,R^1\pi_*(\pi^*\Op(1))\big )\simeq H^0(\pr^1,\Op)\simeq\mathbb{C}$.
Denote by $R_t\subset Y$ the fibre over $t\in\pr^1$. Then the natural map
\[
R^1\pi_*(\pi^*\Op(1))\otimes\mathbb{C}(t)\to H^1(R_t,\ol_{R_t})\simeq \Ext(\ol_{R_t},\ol_{R_t})
\]
is an isomorphism, see for example \cite[III, Corollary 12.9]{MR0463157}. Therefore, the non-zero element $s\in\Ext\big (\mathcal{O}_Y(K_Y),\ol_Y\big )$ corresponds to a non-zero element $s_t\in \Ext(\ol_{R_t},\ol_{R_t})$. Thus $\mathcal{E}$ is a non-split extension (\ref{extE}).

\medskip

{\em Step 2.} We show that $-K_X$ is nef.

It is enough to check $-K_X\cdot C\geq 0$ for any integral curve $C\subset D$, as $-K_X\sim 2D+2F$ and $F$ is nef. Let $C\subset D$ be an integral curve. We have
\[
D^2\sim \varphi^*\big ( c_1(\mathcal{V})\big ) \cdot D \sim -D\cdot F
\]
by the Grothendieck relation. Thus
\[
-K_X\cdot C=(2D+2F)\cdot C=(2D+2F)|_D\cdot C = 0.
\]

\medskip

{\em Step 3.} It remains to show that $-K_X$ is not semi-ample and that $2D$ is the fixed divisor of $|{-}K_X|$.

Since $(-K_X)^2\sim(2D+2F)^2\sim (-4D\cdot F+ 8D\cdot F)=4D\cdot F$
is not numerically zero and \[(-K_X)^3=8(D+F)\cdot D\cdot F=0,\] one has $\nu(X,-K_X)=2$.

Since $2D|_F\sim -K_F$ by the adjunction formula, and $\kappa(F,-K_F)=0$, we obtain \[|{-}mK_X|=2mD+|2mF|\] for any integer $m\geq 1$.
Thus, $2D$ is the fixed divisor of $|{-}K_X|$ and
\[\kappa(X,-K_X)=\kappa(X,F)=1,\]
where the last equality follows from $\ol_X(F)\simeq f^*\Op(1)$.
\end{proof}

\section{The cone of effective curves}\label{sec:cone}
In this section, we consider the threefold $X$ defined as in Theorem \ref{mainnonrat} and we study the $K_X$-trivial curves. We will describe the cone of curves $\overline{\NE}(X)$ and prove Proposition \ref{maincone}. In the whole section, we consider $X$ defined as follows.

\begin{set}\label{setup:ex}
    Let $Y$ be a minimal rational elliptic surface $\pi\colon Y\to\pr^1$, i.e.\ $Y$ is isomorphic to the blow-up of $\pr^2$ at the $9$ base points of a cubic pencil. Let $\mathcal{V}$ be a rank-$2$ vector bundle over $Y$ defined by a non-split extension
    \[   0\to\mathcal{O}_Y\to\mathcal{V}\to\mathcal{O}_Y(K_Y)\to 0
    \]
    and let $\varphi\colon X\coloneqq \pr(\mathcal{V})\to Y$.  

    Then $-K_X$ is nef and  not semi-ample,
    \[|{-}K_X|=2D+|2F|,\] 
    where
 \begin{itemize}[leftmargin=*]
\item $D\simeq Y$ is a section of the $\pr^1$-bundle $\varphi\colon X\to Y$ corresponding to the quotient $\mathcal{V}\to\mathcal{O}_Y(K_Y)\to 0$,
\item $F$ is a general fibre of the fibration $f\coloneqq \pi\circ\varphi\colon X\to\pr^1$, thus is a $\pr^1$-bundle over the general fibre of the elliptic fibration $\pi\colon Y\to \pr^1$.
\end{itemize}
\end{set}

\begin{lem}\label{lem:KXtrivialonD}
    In Setup \ref{setup:ex}, we have $-K_X|_D = 0$.
\end{lem}
\begin{proof}
    By the adjunction formula, $-K_D\sim -K_X|_D - D|_D$, i.e.\ $F|_D\sim (2D+2F)|_D - D|_D$. Hence $D|_D\sim -F|_D$ and we have $-K_X|_D = 0$.
\end{proof}

\begin{rem}\label{rem:coneofD}
{\rm Since we want to locate all the $K_X$-trivial curves in the threefold $X$ obtained in Theorem \ref{mainnonrat}, in view of the above lemma, we describe here all the extremal rays of the cone of curves $\overline{\NE}(D)$, see \cite[Theorem 8.2]{MR2457523}. By the Cone Theorem, the subcone $\NE(D)\cap K_D^{<0}$ is closed and all the $K_D$-negative extremal rays of $\overline{\NE}(D)$ are spanned by $(-1)$-curves. Note that there are infinitely many $(-1)$-curves on the surface $D$. Now as $-K_D$ is nef, it remains to describe $\overline{\NE}(D)\cap K_D^{\perp}$. Since all extremal rays of $\overline{\NE}(D)\cap K_D^{\perp}$ are spanned by either $-K_D$ or a $(-2)$-curve, and there are only finitely many $(-2)$-curves (they are the irreducible components of reducible fibres of $f|_D\colon D\to \pr^1$), the cone  $\NE(D)\cap K_D^{\perp}$ is rational polyhedral. We conclude that the cone $\NE(D)$ is closed and that every extremal ray of $\NE(D)$ is spanned by either a $(-1)$-curve, or a $(-2)$-curve (or $-K_D$, if the elliptic fibration $f|_D\colon D\to\pr^1$ has no reducible fibre).}
\end{rem}

 \begin{lem}\label{lem:Ktrivcurve}
     In Setup \ref{setup:ex}, the cone $\NE(X)\cap K_X^{\perp}$ is closed and every extremal ray of this cone is spanned by one of the following curves:
     \begin{itemize}[leftmargin=*]
     \item a $(-1)$-curve on $D$:
     in this case the contraction of the extremal ray is a simple flopping contraction, i.e.\ the flopping curve has normal bundle isomorphic to $\Op(-1)^{\oplus 2}$;
     \item a $(-2)$-curve on $D$, i.e.\ an irreducible component of a reducible fibre of the elliptic fibration $f|_D\colon D\to \pr^1$:
     in this case the contraction of the extremal ray contracts a divisor isomorphic to $\pr^1\times\pr^1$ to a curve;
     \item a smooth elliptic fibre on $D$ if the elliptic fibration $f|_D\colon D\to \pr^1$ has no reducible fibre:
     in this case there is no smooth rational curve class in the extremal ray.
     \end{itemize}
 \end{lem}
 \begin{proof}
Consider the morphism
\[
\phi\colon \NE(D)\to \NE(X)
\]
induced by the inclusion $D\hookrightarrow X$. 
Since the cone $\NE(D)$ is closed by Remark \ref{rem:coneofD}, we will show that $\phi\big ( \NE(D)\big ) = \NE(X)\cap K_X^{\perp}$, which implies that the cone $\NE(X)\cap K_X^{\perp}$ is also closed.

We will follow the same strategy as in the proof of \cite[Lemma 2.6]{MR2384898}.
     Let $\gamma$ be an integral $K_X$-trivial curve, then
     \begin{equation}\label{eq:Ktrivcurve}
         (2D+2F)\cdot \gamma =0.
     \end{equation}
     Let $\gamma'=\varphi(\gamma)\subset Y$ and let $S\coloneqq \varphi^{-1}(\gamma')$. Then $S$ is a $\pr^1$-bundle over the curve $\gamma'$. Denote by $\ell$ a fibre of $\varphi|_S\colon S\to\gamma'$.

     \medskip
     {\em Case 1.}
     If $F\cdot\gamma = 0$, then $\gamma$ is $f$-vertical and $D\cdot \gamma =0$ by (\ref{eq:Ktrivcurve}). 
     
     In this case, $\gamma'$ is either an integral fibre of the ellitpic fibration $\pi\colon Y\to\pr^1$, or an irreducible component of some reducible fibre, i.e.\ a $(-2)$-curve on $Y$. Denote by $\tilde{\gamma}$ the integral curve $D\cap S$, then $\tilde{\gamma}\simeq \gamma'$ as $\varphi|_D\colon D\to Y$ is an isomorphism. We will show that 
     $[\gamma]\in\R_+[\tilde{\gamma}]$ in this case.

\medskip
     (i) If $\gamma'$ is a $(-2)$-curve on $Y$, then 
      \[
     \mathcal{V}|_{\gamma'}\simeq \Op\oplus \Op
     \]
and thus $S=\pr(\mathcal{V}|_{\gamma'})\simeq \pr^1\times \pr^1$.  We have
    \[S\cdot \gamma = \varphi^*(\gamma')\cdot \gamma = \gamma'\cdot \varphi(\gamma) = \gamma'^2 = -2.\]
     By the adjunction formula, $-K_S\cdot \gamma = -K_X\cdot \gamma - S\cdot\gamma = 2$ and thus $\gamma$ is the ruling of $S\simeq \pr^1\times \pr^1$ other than $\ell$.
     Since $S$ is not nef, by \cite[Lemma 2.5]{Xie20}, there exists a $(K_X+S)$-negative extremal ray $\Gamma$ such that $S\cdot \Gamma <0$. Moreover, $S\simeq \pr^1\times \pr^1$ and $S\cdot \ell = 0$, we deduce that $\Gamma=\R_+[\gamma]$ and the contraction of $\Gamma$ contracts $S$ to a curve.

     Since $\tilde{\gamma}\subset S\simeq \pr^1\times\pr^1$, $-K_X\cdot \tilde{\gamma} = -K_X\cdot \gamma=0$ and $-K_X\cdot \ell>0$, we obtain that $[\gamma]$ and $[\tilde{\gamma}]$ are proportional. More precisely, since $S\cdot \tilde{\gamma} = (S|_D)^2 = -2 = S\cdot \gamma$, we have $[\gamma]=[\tilde{\gamma}]$.

\medskip
     (ii) If $\gamma'$ is an integral fibre (i.e.\ a smooth elliptic curve or a singular rational curve) of $\pi\colon Y\to\pr^1$, then $S\sim F$ and $-K_S\sim (-K_X-S)|_S\sim 2D|_S$ is an effective and $f$-vertical $1$-cycle. Since in Setup \ref{setup:ex}, the fixed divisor of $|{-}K_X|$ has no $f$-vertical part, we have that $D$ is $f$-relatively nef and we can apply Lemma \ref{verticalcurveinAh} with $A_h=2D$. Then $D\cdot \tilde{\gamma} =0$. Thus $(-K_S)^2 =0 $ and $-K_S$ is nef. Since $\rho(S)$=2 and $-K_S\cdot\ell > 0$, we have $[\gamma]\in\mathbb{R}_+[\tilde{\gamma}]$.
     
     Note that there is no smooth rational curve contained in $S$, whose class is in $\mathbb{R}_+[\tilde{\gamma}]$: assume that there exists such a rational curve $\gamma_{rat}\simeq \pr^1$. Since $-K_X\cdot \gamma_{rat} =0 $, the curve $\gamma_{rat}$ is not contracted by $\varphi$ and thus $\varphi|_{\gamma_{rat}}$ maps $\gamma_{rat}$ surjectively to its image. But $\varphi(\gamma_{rat})\subset\varphi(S)=\gamma'$ is a smooth elliptic curve or an integral non-normal rational curve, this cannot happen. 

     Therefore, if the elliptic fibration $\pi\colon Y\to\pr^1$ has no reducible fibre (and thus the elliptic fibration $f|_D\colon D\to\pr^1$ has no reducible fibre), then every fibre of $f|_D\colon D\to\pr^1$ is in the same class $[\tilde{\gamma}]$ and there is no smooth rational curve whose class is in $\mathbb{R}_+[\tilde{\gamma}]$. Note moreover that in this case, the class $[-K_D]=[\tilde{\gamma}]\in\NE(X)$ is not a non-negative linear combination of the classes of $(-1)$-curves on $D$. This is because by considering the nef divisor $F\subset X$, the curve $-K_D$ is $F$-trivial, but any $(-1)$-curves on $D$ is $F$-positive.

\medskip
     {\em Case 2.}
     If $F\cdot \gamma > 0$, then $\gamma$ is $f$-horizontal and $D\cdot \gamma <0$ by (\ref{eq:Ktrivcurve}). Thus $\gamma\subset D$. By Lemma \ref{lem:KXtrivialonD} and Remark \ref{rem:coneofD}, the curve $\gamma$ is a non-negative linear combination of $(-1)$-curves and $(-2)$-curves (or $-K_D$) as curve classes. 
     Now by {\em Case 1}, it suffices to show that the ray spanned by the class a $(-1)$-curve on $D$ is an extremal ray of $\overline{\NE}(X)$ and to describe the corresponding contraction of this extremal ray.
     
     Assume that $\gamma$ is a $(-1)$-curve on $D$. Since $\varphi|_D\colon D\to Y$ is an isomorphism, $\gamma'$ is also a $(-1)$-curve on $Y$. Hence,
     \[
     \mathcal{V}|_{\gamma'}\simeq \Op\oplus \Op(-1)
     \]
and thus $S=\pr(\mathcal{V}|_{\gamma'})\simeq\mathbb{F}_1$. 
     We also have
    \[S\cdot \gamma = \varphi^*(\gamma')\cdot \gamma = \gamma'\cdot \varphi(\gamma) = \gamma'^2 = -1.\]
     By the adjunction formula, $-K_S\cdot \gamma = -K_X\cdot \gamma - S\cdot\gamma = 1$ and thus $\gamma$ is the section with minimal self-intersection number of the ruled surface $S\simeq \mathbb{F}_1$. Therefore, the normal bundle sequence
    \[
0 \to N_{\gamma/S} \to N_{\gamma/X} \to N_{S/X}|_{\gamma} \to 0
    \]
    splits, and we obtain $N_{\gamma/X}\simeq \Op(-1)^{\oplus 2}$.

     Since $S$ is not nef, by \cite[Lemma 2.5]{Xie20}, there exists a $(K_X+S)$-negative extremal ray $\Gamma$ such that $S\cdot \Gamma <0$. Moreover, $S$ is a ruled surface and $S\cdot \ell =0$, we deduce that $\Gamma=\R_+[\gamma]$. Since $S\cdot \gamma <0$ and $\gamma$ is the minimal section of $S\simeq\mathbb{F}_1$, the contraction of $\Gamma$ is small and contracts precisely the curve $\gamma$. The corresponding flop blows up the curve $\gamma$ in $X$ with exceptional divisor isomorphic to $\pr^1\times\pr^1$, and blows down the exceptional divisor in the other direction onto some smooth threefold $X^+$.

     \medskip
     
     Combining the above two cases, we have shown that the class of any effective $K_X$-trivial curve is a non-negative linear combination of the classes of $(-1)$-curves and $(-2)$-curves (or $-K_D$) on $D$. Together with Remark \ref{rem:coneofD}, we obtain $\phi\big ( \NE(D)\big ) = \NE(X)\cap K_X^{\perp}$. 
 \end{proof}

\begin{proof}[Proof of Proposition \ref{maincone}]
 Since $\NE(X)\cap K_X^{<0}$ is closed by the Cone Theorem and $\NE(X)\cap K_X^{\perp}$ is closed by Lemma \ref{lem:Ktrivcurve}, the cone $\NE(X)$ is closed. The other statements follow from Lemma \ref{lem:Ktrivcurve}.
 \end{proof}

 \begin{rem}
     {\rm In Proposition \ref{maincone}, the cone $\NE(X)$ has a unique $K_X$-negative extremal ray spanned by the class of a fibre of the $\pr^1$-bundle $X=\pr(\mathcal{V})\to Y$. This is because $-K_X$ is divisible by two in $\Pic(X)$, which implies that the only possible birational $K_X$-negative contraction contracts a divisor to a smooth point by the classification of Mori, see \cite[Section 3]{MR715648}; then we conclude by Propositions \ref{generalbiratcontr} and \ref{generalnonbiratcontr}.}
 \end{rem}
 
\bibliographystyle{alpha}
  \bibliography{bibliography.bib}

    \Affilfont{\small{Z\textsc{hixin} X\textsc{ie}, F\textsc{achrichtung} M\textsc{athematik}, C\textsc{ampus}, G\textsc{eb\"{a}ude} E2.4, U\textsc{niversit\"{a}t} \textsc{des} S\textsc{aarlandes}, 66123 S\textsc{aarbr\"{u}cken}, G\textsc{ermany} }}

\textit{Email address:}
\href{mailto:xie@math.uni-sb.de}{xie@math.uni-sb.de}
\end{document}